\documentclass[10pt,a4paper]{article}
\usepackage[utf8]{inputenc}
\usepackage[english]{babel}
\usepackage[T1]{fontenc}
\usepackage{amsmath}
\usepackage{amsthm}
\usepackage{dsfont}
\usepackage{amsfonts}
\usepackage{amssymb}
\usepackage{graphicx}
\usepackage{tabularx}
\usepackage{lmodern}
\usepackage[colorlinks=true,linkcolor=blue]{hyperref}
\usepackage{relsize}
\usepackage{geometry}
\usepackage{mathtools}

\newtheoremstyle{theorema}%
    {4pt}
    {4pt}
    {\slshape}
    {}
    {\bfseries}
    {.}
    {.5em}
    {}
\theoremstyle{theorema} 
\newtheorem{theorem}{Theorem}

\newtheorem{proposition}[theorem]{Proposition}
\newtheorem{lemma}[theorem]{Lemma}

\newtheorem{corollary}[theorem]{Corollary}

\theoremstyle{definition} 
\newtheorem{definition}[theorem]{Definition}

\theoremstyle{remark} 
\newtheorem*{remark}{Remark}

\DeclareMathOperator*{\esssup}{ess\,sup}

\begin{document}

\title{On the number of eigenvalues of modified permutation matrices in mesoscopic intervals}

\author{Valentin B\textsc{ahier} \thanks{valentin.bahier@math.univ-toulouse.fr, Institut de Math\'ematiques de Toulouse, 118 route de Narbonne, F-31062 Toulouse Cedex 9, France.}}

\date{\today}

\maketitle

\begin{abstract}
We are interested in two random matrix ensembles related to permutations: the ensemble of permutation matrices following Ewens' distribution of a given parameter $\theta >0$, and its modification where entries equal to $1$ in the matrices are replaced by independent random variables uniformly distributed on the unit circle. For the elements of each ensemble, we focus on the random numbers of eigenvalues lying in some specified arcs of the unit circle. We show that for a finite number of fixed arcs, the fluctuation of the numbers of eigenvalues belonging to them is asymptotically Gaussian. Moreover, for a single arc, we extend this result to the case where the length goes to zero sufficiently slowly when the size of the matrix goes to infinity. Finally, we investigate the behaviour of the largest and smallest spacing between two distinct consecutive eigenvalues.
\end{abstract}

\section{Introduction}

\subsection{Random permutation matrices}

The spectrum of random permutation matrices has drawn much attention the last few decades. On the one hand, working with matrices brings a new approach for understanding the structure of permutation groups. On the other hand, the sets of permutation matrices can be seen as finite subgroups of orthogonal groups or unitary groups, and thus their studies give the opportunity to see how much of the structure of larger groups emerges through these finite subgroups.

To make it clear, let us recall the few following definition and facts: \\
A \textbf{permutation matrix} is a square matrix that has exactly one entry equal to $1$ in each row and each column and $0$'s elsewhere. Since such a matrix is in particular unitary, its spectrum is included in the unit circle. There is a correspondence between $\mathfrak{S}_N$ (the set of permutations of order $N$) and the set of permutation matrices of size $N$. The spectrum of any permutation matrix is completely determined by the \textbf{cycle structure} of its corresponding permutation. In other words, this spectrum is a function of the numbers of cycles of same size when one decomposes the permutation into disjoint cycles. Besides, when a permutation is uniformly chosen at random, the joint distribution of these cycle counts is known (see \cite{arratia2003logarithmic}, Chapter 1). Wieand took advantage of this to investigate the asymptotic behaviour of the counting function for the eigenvalues of permutation matrices lying in some fixed arcs on the unit circle \cite{wieand2000eigenvalue}, and also for some wreath products involving $\mathfrak{S}_N$ \cite{wieand2003permutation}, under the uniform distribution.
In addition, the work of Blair-Stahn \cite{blair2000random} revealed how difficult it is to compute the limiting expectation of the counting function for a shrinking interval of type $\left( \mathrm{e}^{2i\pi a } , \mathrm{e}^{2i\pi \left( a+\frac{b}{N} \right)} \right]$ on the unit circle, with $a,b$ fixed real numbers, and in particular for some special $a$ there is not known explicit expression.
 
The uniform distribution on $\mathfrak{S}_N$ has not been the only one studied. Indeed, the use of Ewens measures is quite natural \cite{ewens1972sampling} and very convenient to study in some mathematical aspects (see \cite{arratia2003logarithmic} and \cite{chafai2013processus} for a definition and related results).
Heuristically, the Ewens measures are one-parameter deformations of the uniform distribution, where the parameter (usually denoted by $\theta >0$) influences the expected total number of cycles in the decomposition into disjoint cycle of a randomly chosen permutation. Ben Arous and Dang \cite{ben2015fluctuations} tackled this family of measures over permutation matrices and gave some asymptotic results for linear statistics of their spectrum (not only the counting function).

A classical motivation for the study of Ewens measures can be found in population genetics, where the distribution of the $n$-tuple whose the $i$-th coordinate counts the number of alleles represented $i$ times in a random sample of $n$ gametes (taken from a population under certain conditions), is exactly the distribution obtained considering the integer partition induced by a permutation under a certain Ewens measure on $\mathfrak{S}_n$. In this framework the parameter $\theta$ of the considered Ewens measure plays the role of a \emph{population mutation rate}.
(see Ewens'~sampling formula \cite{ewens1972sampling})

Furthermore, wreath products have some applications in group theory (\emph{e.g.} finding the exhaustive list of Sylow groups from a given finite permutation group) and in graph theory (\emph{e.g.} looking at automorphisms on regular rooted trees (see for example \cite{evans2002eigenvalues})). Basically, introduction of randomness in these ensembles aims to have a better understanding of their structure. 

Before stating in which way we want to extend the results of Wieand and Blair-Stahn in this article, let us mention some other relative work, such as the study of characteristic polynomial of random permutation matrices by Hambly, Keevash O'Connell, Stark \cite{hambly2000characteristic} or of more generally multiplicative class functions for some wreath products by Zeindler \emph{et al.} \cite{zeindler2010permutation} \cite{zeindler2013central} \cite{dang2014characteristic} \cite{dehaye2013averages}. We can also mention the various results of Najnudel and Nikeghbali \cite{najnudel2013distribution} for the point process of eigenvalues where a meaning of almost sure convergence of the empirical spectral measure is made precise for some modified random permutation matrices.

Now, we introduce the way we shall continue some of the previous works:
\begin{itemize}
	\item We look at the counting function of eigenvalues for the ensemble of permutation matrices and the wreath product $S^1 \wr \mathfrak{S}_N$ (where $S^1$ is the group of complex numbers of modulus $1$) endowed with Ewens measures. Our motivation for studying this particular wreath product is twofold: its spectral distribution is quite more convenient to study, and it brings closer the analogy with the Circular Unitary Ensemble. Indeed, in contrast with the ensemble of permutation matrices, the distribution of eigenvalues for $S^1 \wr \mathfrak{S}_N$ is invariant by rotation. 
	\item We take advantage of some tools introduced in the articles of Wieand \cite{wieand2000eigenvalue} and Ben Arous, Dang \cite{ben2015fluctuations} and develop them in our framework.
	\item We also investigate the counting function at an intermediate scale (mesoscopic) between macroscopic and microscopic scales, where the observed number of eigenvalues still tends to infinity when $N$ goes to infinity. The study of its fluctuations is motivated by comparison with an analogous result of Bourgade which is given as a consequence of Theorem~$1.4$ in \cite{bourgade2010mesoscopic}, for unitary matrices.
\end{itemize}

\subsection{Notations and main results}

For all real numbers $x$, we denote by $\lfloor x \rfloor$ the floor of $x$, $\lceil x \rceil$ the ceiling of $x$, and $\{x\} = x - \lfloor x \rfloor$ the fractional part of $x$. \\
If $(u_n)$,$(v_n)$ are sequences of real numbers such that $(v_n)$ is positive, and if $x$ is a real number, we will write $u_n = \mathcal{O}_x (v_n)$ or $u_n \ll_x v_n$ when there exists a constant $C_x$ such that for all $n$, $\vert u_n \vert \leq C_x v_n$. \\
Let $\theta >0$. Let $\left( \sigma_N \right)_{N \geq 1}$ be a sequence of random permutations following Ewens measure of parameter $\theta$. 
Formally, it means that for all $N$, $\sigma_N$ takes values in $\mathfrak{S}_N$ and 
\[\forall \sigma \in \mathfrak{S}_N, \ \mathbb{P} (\sigma_N = \sigma ) = \mathbb{P}_\theta^{(N)} (\sigma ) = \frac{\theta^{K(\sigma)}}{\theta (\theta + 1) \cdots (\theta + N-1 )}\]
where $K(\sigma)$ denotes the total number of cycles of $\sigma$ once decomposed into disjoint cycles.
Let $\left( z_j \right)_{j \geq 1}$ be a sequence of i.i.d random variables uniformly distributed on the unit circle, independent of $\left( \sigma_N \right)_{N \geq 1}$. For all $N \geq 1$, we define $M_N$ and $\widetilde{M}_N$ as the $N$-by-$N$ matrices whose entries are given by:
\[\forall 1\leq i,j \leq N , \quad \left\{ \begin{array}{rl}
(M_N)_{i,j} &:= \mathds{1}_{i = \sigma_N (j)  } \\
(\widetilde{M}_N)_{i,j} &:= z_i\mathds{1}_{i = \sigma_N (j) }.
\end{array}\right.   \]
In all the following we identify the ensemble of permutation matrices of order $N$ and the symmetric group $\mathfrak{S}_N$, and consider the outputs of $\widetilde{M}_N$ as elements of the wreath product of $S^1$ and $\mathfrak{S}_N$, denoted by $S^1 \wr \mathfrak{S}_N$.  

The elements of both ensembles are in particular unitary matrices, and thus their eigenvalues belong to the unit circle. 

Then, the question of the asymptotic behaviour of the distribution of these eigenvalues arises naturally, in particular if one wants to compare them with some known results on other random matrix ensembles.

To this purpose, let $I:= \left( \mathrm{e}^{2i\pi \alpha} , \mathrm{e}^{2i\pi \beta} \right]$ the interval which denotes the arc on the unit circle from $\mathrm{e}^{2i\pi \alpha}$ (excluded) to $\mathrm{e}^{2i\pi \beta}$ (included), with $0\leq \alpha <1$ and $\alpha < \beta \leq \alpha +1$. We take it half-open for practical reason. The conditions on $\alpha$ and $\beta$ are sufficient to take whatever (half-open) interval of the unit circle. For $N\geq 1$, we define $X_N^I$ and $\widetilde{X}_N^I$ as the respective numbers of eigenvalues of $M_N$ and $\widetilde{M}_N$ lying in $I$. 

\begin{lemma}
Let $s,t,u,v \in \mathbb{R}$. The following limits exist, are finite, and can be explicitly computed.
\[c (s, t, u, v) :=\lim_{N\to \infty} \frac{1}{N} \sum_{j=1}^N  (\{js\} - \{jt\} )(\{ju\} - \{jv\} ).\]
\[\widetilde{c} (s, t, u, v) := \lim_{N\to \infty} \frac{1}{2N} \sum_{j=1}^N  (h_j (t - u) + h_j (s - v) - h_j (s - u) - h_j (t - v))\]
with $h_j (x) := \{jx \}(1- \{jx\})$.
\end{lemma}
We refer to Lemma~\ref{lem:Hcesaro} for a condensed version of this result and a proof, inspired from \cite{wieand2000eigenvalue} for the first limit and \cite{wieand2003permutation} for the second one.

The following theorem have already been established in \cite{wieand2000eigenvalue} and \cite{wieand2003permutation}, in the particular case $\theta=1$. Also, for $m=1$ and for random permutation matrices without modification, the result simply derives from Theorem~$1.5$ in \cite{ben2015fluctuations}. Furthermore, the third item can be deduced from Proposition~$1.2$ in \cite{dang2014characteristic} considering the imaginary part of the logarithm of the characteristic polynomial, for the specific case where the family $(1,\alpha_1, \cdots , \alpha_m, \beta_1, \cdots , \beta_m)$ is linearly independent over $\mathbb{Z}$.

\everymath{\displaystyle}
\begin{theorem}\label{th:1}
Let $I_1 , \cdots , I_m$ be a finite number of fixed arcs of the form $I_k := \left(\mathrm{e}^{2i\pi \alpha_k} , \mathrm{e}^{2i\pi \beta_k} \right]$. For $1\leq k,l \leq m$, denote $c_{k,l} := c (\alpha_k, \beta_k , \alpha_l , \beta_l)$ and $\widetilde{c}_{k,l} := \widetilde{c} (\alpha_k, \beta_k , \alpha_l , \beta_l )$. Then, as $N \to \infty$
\begin{enumerate} 
	\item \[\mathrm{Var} (X_N^{I_k}) \sim  c_{k,k} \theta \log N ,\quad \mathrm{Var} (\widetilde{X}_N^{I_k}) \sim  \widetilde{c}_{k,k} \theta \log N. \]
	\item \[\left( \frac{X_N^{I_1} - \mathbb{E} ( X_N^{I_1}) }{\sqrt{\mathrm{Var} (X_N^{I_1})}} , \cdots , \frac{X_N^{I_m} - \mathbb{E} ( X_N^{I_m}) }{\sqrt{\mathrm{Var} (X_N^{I_m})}}  \right) \overset{\text{d}}{\longrightarrow} \mathcal{N}(0,D)\]
	where $D=(D_{k,l})_{1\leq k , l \leq m}$ with $D_{k,l}= \frac{c_{k,l}}{\sqrt{c_{k,k} c_{l,l}}}$.
	\item  \[\left( \frac{\widetilde{X}_N^{I_1} - \mathbb{E} ( \widetilde{X}_N^{I_1}) }{\sqrt{\mathrm{Var} (\widetilde{X}_N^{I_1})}} , \cdots , \frac{\widetilde{X}_N^{I_m} - \mathbb{E} ( \widetilde{X}_N^{I_m}) }{\sqrt{\mathrm{Var} (\widetilde{X}_N^{I_m})}}  \right) \overset{\text{d}}{\longrightarrow} \mathcal{N}(0,\widetilde{D}) \]
	where $\widetilde{D}=(\widetilde{D}_{k,l})_{1\leq k , l \leq m}$ with $\widetilde{D}_{k,l}= \frac{\widetilde{c}_{k,l}}{\sqrt{\widetilde{c}_{k,k} \widetilde{c}_{l,l}}}$.	
\end{enumerate}
\end{theorem}
 
Most of the innovative work in this paper holds in the following main result.
 
\begin{theorem}
\label{th:2}
Assume $I$ to be depending on $N$, of the form $I=I_N:=  \left(\mathrm{e}^{2i\pi \alpha_N} , \mathrm{e}^{2i\pi \beta_N} \right]$. \\ Denote $\delta_N := \beta_N - \alpha_N >0$. Suppose that the sequence $(\delta_N)$ satisfies \[\left\{\begin{array}{l} 
\delta_N \underset{N\to \infty}{\longrightarrow} 0 \\
N \delta_N \underset{N\to \infty}{\longrightarrow} + \infty .
\end{array} \right.\]
\begin{enumerate}
	\item Then, as $N \to \infty$
\[  \mathrm{Var} (\widetilde{X}_N^{I}) \sim \frac{\theta}{6} \log(N\delta_N) \] and
\[ \frac{\widetilde{X}_N^I - \mathbb{E} ( \widetilde{X}_N^I) }{\sqrt{\mathrm{Var} (\widetilde{X}_N^I)}}  \overset{\text{d}}{\longrightarrow} \mathcal{N}(0,1) .\]
	\item Suppose in addition that the sequence $(\alpha_N)$ is constant, say $\alpha_N=\alpha$ for all $N$. Then, as $N \to \infty$
\[\mathrm{Var} (X_N^I) \sim   
\left\{\begin{array}{ll}
\frac{\theta}{6} \log(N\delta_N)  & \text{if } \alpha \text{ is irrationnal}\\
\theta \left(\frac{1}{6} + \frac{1}{6q^2}\right) \log(N\delta_N) & \text{if } \alpha=\frac{p}{q} \text{ with } p,q \text{ coprime integers}
\end{array}\right.	\]
and
\[ \frac{X_N^I - \mathbb{E} (X_N^I) }{\sqrt{\mathrm{Var} (X_N^I)}}  \overset{\text{d}}{\longrightarrow} \mathcal{N}(0,1) .\]
\end{enumerate}
\end{theorem}
\everymath{}

The article is organized as follows: In section~\ref{sec:2} we begin with preliminary results about Cesàro means and Feller Coupling. In sections~\ref{sec:3} and \ref{sec:4} we prove Theorems~\ref{th:1} and \ref{th:2}, investigating the asymptotic behaviour of the mean and variance of the considered sequences of random variables. In section~\ref{sec:5} we look at the extremal spacings between two consecutive eigenvalues and establish some results of tightness. This last section is independent of the sections~\ref{sec:3} and \ref{sec:4}.

\section{Preliminaries}
\label{sec:2}

\subsection{Cesàro means of fractional order}

We set up here a few results (highly inspired from \cite{zygmund2002trigonometric} Volume 1 chapter 3, and \cite{ben2015fluctuations}) about Cesàro means of fractional order, that we will use to investigate the asymptotical behaviour of the variance.

\begin{definition}
For all $1\leq j\leq n$,
\begin{equation}\label{eq:defPsi}
\Psi_n (j) := \frac{n(n-1) \cdots (n-j+1) }{(\theta + n-1) \cdots (\theta + n-j)}. 
\end{equation}
\end{definition}

\begin{definition}
\begin{enumerate}
	\item The \textbf{Cesàro numbers} of order $\delta \in \mathbb{R}\setminus \{-1,-2,\cdots \}$ are given by \[A_n^\delta := \binom{n+\delta}{n} = \frac{(n+\delta) \cdots (1+ \delta)}{n!}.  \]
	\item The \textbf{Cesàro mean} of order $\theta >0$ of the sequence $w=(w_j)_{j\geq 0}$ is given by \[\sigma_n^\theta (w) := \sum_{j=0}^n \frac{A_{n-j}^{\theta -1}}{A_n^{\theta}} w_j. \]
	\item A sequence of real numbers $w=(w_j)_{j\geq 0}$ is said to be \textbf{convergent in Cesàro sense} of order $\theta$ (and will be denoted by $(C, \theta)$) to a limit $\ell$ iff $\sigma_n^\theta (w) \underset{n\to \infty}{\longrightarrow} \ell$.
\end{enumerate}
\end{definition}

\begin{remark}
If $w_0=0$, then the Cesàro mean of the sequence $(w_j)$ can be reformulated as
\begin{equation}
\sigma_n^\theta (w) = \sum_{j=1}^n \frac{A_{n-j}^{\theta -1}}{A_n^{\theta}} w_j = \frac{\theta}{\theta + n} \sum_{j=1}^n \Psi_n (j) w_j.  
\end{equation}
\end{remark}

\begin{lemma}[\cite{zygmund2002trigonometric} Vol I, page 74]
\label{lem:regular}
Let $M=(M_{i,j})_{\substack{1\leq j\leq i}}$ a lower-triangular infinite stochastic matrix satisfying 
\[\forall j \geq 1, \ \lim_n M_{n,j} = 0. \]
Let $s=(s_1, \dots , s_n , \dots )^T \in \mathbb{R}^\mathbb{N}$ and let $t=Ms$. If $s_n \rightarrow L$, then $t_n \rightarrow L$.
\end{lemma}

The next proposition is a particular case of the Lemma 2.27 from \cite{zygmund2002trigonometric} Vol II page 70.

\begin{proposition}\label{prop:cesaro}
If a sequence $(w_n)$ of real numbers is bounded and converging $(C,1)$ to a real number $\ell$, then it converges $(C,\delta)$ to $\ell$ for all $\delta >0$.
\end{proposition}

From all this we can deduce the following results involving the numbers $\Psi_{n,j}$:

\begin{proposition}\label{prop:log}
Let $(w_j)_{j\geq 1}$ be a sequence of non-negative real numbers. Suppose that 
\[\lim_{n \to \infty} \frac{1}{n} \sum\limits_{j=1}^n \Psi_n (j) w_j = L >0.\]
Then 
\[\lim_{n \to \infty} \frac{1}{\log n} \sum\limits_{j=1}^n \frac{\Psi_n (j)}{j} w_j = L \theta .\]
\end{proposition}
\begin{proof}
For all $n\geq 1$, let us define $s_n:= \frac{1}{n} \sum\limits_{j=1}^n \Psi_n (j) w_j$, and $t_n:= \frac{1}{J_n} \sum\limits_{j=1}^n \frac{\Psi_n (j)}{j} w_j$, where $J_n:= \sum\limits_{j=1}^n \frac{\theta}{\theta + j-1}$. We introduce the infinite matrix $M=(M_{n,j})$ defined for all $n,j \geq 1$ by
\begin{equation}
M_{n,j} = \frac{1}{J_n} \left( \frac{\theta}{\theta +j} \mathds{1}_{j<n} + \mathds{1}_{j=n} \right) . 
\end{equation}
The conditions of Lemma~\ref{lem:regular} are easy to check for the matrix $M$. \\
We show $t_n = \sum_{j=1}^n M_{n,j} s_j$ by induction:
\begin{itemize}
	\item \underline{$n=1$}: $t_1 = \Psi_1 (1) w_1 = \frac{w_1}{\theta}$ and $\sum_{j=1}^1 M_{1,j} s_j = s_1 = \frac{w_1}{\theta}$.
	\item \underline{$n-1$ to $n$}: Suppose that the statement holds for $n-1$. We prove it at the step $n$, in other words we want to show
	\begin{equation}\label{eq:recurrencePsi}
	\sum_{j=1}^n \frac{\Psi_n (j)}{j} w_j = \sum_{j=1}^{n-1} \frac{\theta}{\theta + j} s_j +s_n 
	\end{equation}
	From \eqref{eq:defPsi} we have 
	\[ \Psi_n (j) \left[ \frac{1}{j} - \frac{1}{n} \right] = \Psi_{n-1} (j) \left[ \frac{1}{j} - \frac{1}{\theta + n -1} \right] \]
	for $1\leq j\leq n-1$ and then, successively,  
	\begin{align*}
	\frac{\Psi_n (j)}{j} &= \frac{\Psi_{n-1} (j)}{j} -\frac{\Psi_{n-1} (j)}{\theta + n -1}  + \frac{\Psi_n (j)}{n} \\
	\sum_{j=1}^{n-1} \frac{\Psi_n (j)}{j} w_j &= \sum_{j=1}^{n-1} \frac{\Psi_{n-1} (j)}{j} w_j - \frac{1}{\theta + n -1} \sum_{j=1}^{n-1} \Psi_{n-1} (j) w_j + \frac{1}{n} \sum_{j=1}^{n-1} \Psi_n (j) w_j \\
		&= \sum_{j=1}^{n-1} \frac{\Psi_{n-1} (j)}{j} w_j - \frac{n-1}{\theta + n -1} s_{n-1} + s_n - \frac{\Psi_n (n)}{n} w_n
	\end{align*}
	and, applying the induction hypothesis \eqref{eq:recurrencePsi} at step $n-1$, we get 
	\[\sum_{j=1}^{n-1} \frac{\Psi_n (j)}{j} w_j + \frac{\Psi_n (n)}{n} w_n = \sum_{j=1}^{n-2} \frac{\theta}{\theta + j} s_j +s_{n-1} - \frac{n-1}{\theta + n -1} s_{n-1} + s_n = \sum_{j=1}^{n-1} \frac{\theta}{\theta + j} s_j + s_n  \]
	which is \eqref{eq:recurrencePsi} at step $n$.
\end{itemize}
Thus Lemma~\ref{lem:regular} applies and gives $t_n \underset{n\to \infty}{\longrightarrow} L$. \\
Finally, it just remains to see that $J_n / \log n \rightarrow \theta $, which is clear by comparison with the harmonic series.
\end{proof}

\begin{lemma}\label{lem:EqPsi}
For all $n\geq 1$,
\begin{equation}\label{eq:moyPsi}
\frac{1}{n} \sum_{j=1}^n \Psi_n (j) = \frac{1}{\theta}
\end{equation} 
and
\begin{equation}\label{eq:sumPsij}
\sum_{j=1}^n \frac{\Psi_n (j)}{j} = \sum_{j=1}^n \frac{1}{\theta + j -1}. 
\end{equation}
\end{lemma}
\begin{proof}
First, as $ \sum_{k=0}^{+\infty} A_k^\theta x^k = (1-x)^{-\theta -1}  = \frac{1}{1-x} (1-x)^{-(\theta -1)-1}$, then $A_n^\theta = \sum_{j=0}^n A_j^{\theta -1}$. Moreover, it is easy to notice that $\Psi_n (j) = \frac{\theta + n}{\theta} \frac{A_{n-j}^{\theta -1}}{A_n^\theta}$, $1\leq j \leq n$, and $ \frac{\theta + n}{\theta} \frac{A_{n}^{\theta -1}}{A_n^\theta}=1$. Therefore, $1+\sum_{j=1}^n \Psi_n (j) = \frac{\theta + n}{\theta}$, which gives \eqref{eq:moyPsi}.

For \eqref{eq:sumPsij}, we use the same notation as in the proof of Proposition~\ref{prop:log}. We had proven that for all $n$, $t_n = (Ms)_n$. In particular, if $(w_j)$ is the constant sequence equal to $1$, then, following \eqref{eq:moyPsi}, $s_j = \frac{1}{\theta}$ for all $j$, and
\begin{align*}
t_n = \frac{1}{J_n} \sum_{j=1}^n \frac{\Psi_n (j)}{j} = (Ms)_n &=  \frac{1}{J_n} \left(\sum_{j=1}^{n-1} \frac{\theta}{\theta + j}   s_j + s_n \right) \\
	&= \frac{1}{J_n} \left(\sum_{j=1}^{n-1} \frac{1}{\theta + j}  + \frac{1}{\theta} \right) = \frac{1}{\theta}
\end{align*}
so that we have \eqref{eq:sumPsij}.
\end{proof}

\subsection{Feller Coupling}

Let $(\sigma_n)_{n\geq 1}$ be a sequence of random permutation generated under Ewens measure of parameter $\theta$. For all $n\geq 1$ we denote by $a_{n,j}$ the number of $j$-cycles in the decomposition into disjoint cycles of $\sigma_n$. The $a_{n,j}$ are also called the \textbf{cycle counts} of $\sigma_n$.

The next result will be useful to prove the main results of the paper. It consists in an approximation of the cycle counts by independent Poisson random variables, using the so-called \textbf{Feller Coupling} (see \cite{arratia2003logarithmic}).

\begin{lemma}[Lemma 5.3 of \cite{arratia2003logarithmic}]\label{lem:majPoisson}
One can couple $(\sigma_n)_{n\geq 1}$ with a sequence $(W_j)_{j\geq 1}$ of independent Poisson random variables of parameter $\theta/j$ in such a way that 
\[\mathbb{E} \left( \sum_{j=1}^n  \left| a_{n,j} - W_j \right|  \right) = \mathcal{O}_\theta (1).  \]
\end{lemma}

\begin{remark}
A proof of this result is given in \cite{arratia1992poisson} pages 525-526, without any consideration of Cesàro numbers. In the Appendix of the present paper we provide a simple proof involving Cesàro means in order to point out that they naturally emerge from Feller coupling.
\end{remark}

We end preliminaries with two lemmas which will be useful to prove Theorem~\ref{th:2}.
\begin{lemma}
For all $n\geq 1$,
\begin{equation}\label{eq:quadPsi}
\sum_{1\leq j,k \leq n} \frac{1}{jk} \left( \Psi_n (j) \Psi_n (k) - \Psi_n (j+k) \mathds{1}_{j+k\leq n}\right)  = \sum_{k=0}^{n-1} \frac{1}{(\theta+k )^2}. 
\end{equation}
\end{lemma}
\begin{proof}
Denoting $K_n:=\sum\limits_{j=1}^n a_{n,j}$ the total number of cycles, we first notice that
\begin{align*}
\mathrm{Var} (K_n) &= \sum_{j=1}^n \mathrm{Var} (a_{n,j}) + \sum_{\substack{1\leq j,k \leq n \\ j\neq k}} \mathrm{Cov} (a_{n,j} , a_{n,k}) \\
	&= \sum_{1\leq j,k \leq n} \Psi_n (j+k)\mathds{1}_{j+k\leq n} \frac{\theta^2}{jk} + \mathbb{E} (K_n) - \sum_{j=1}^n \left( \frac{\theta}{j} \Psi_n (j) \right)^2 - \sum_{\substack{1\leq j,k \leq n \\ j\neq k}} \frac{\theta}{j} \Psi_n (j) \frac{\theta}{k} \Psi_n (k) \\
	&= \mathbb{E} (K_n) + \theta^2 \sum_{1\leq j,k \leq n} \frac{1}{jk} \left( \Psi_n (j+k) \mathds{1}_{j+k\leq n} - \Psi_n (j) \Psi_n (k)\right)  .   
\end{align*}
Moreover, the Feller Coupling provides the nice expression $K_n= \xi_1 + \cdots + \xi_n$, where the $\xi_k$ are independent Bernoulli variables with parameter $\frac{\theta}{\theta+k-1}$. From this expression of $K_n$ it follows $\mathbb{E} (K_n ) =\sum\limits_{k=0}^{n-1} \frac{\theta}{\theta + k}$ and $\mathrm{Var} (K_n) = \sum\limits_{k=0}^{n-1} \frac{\theta k}{(\theta +k)^2}$. Hence, 
\[\mathbb{E}(K_n) - \mathrm{Var} (K_n) = \sum_{k=0}^{n-1} \left( \frac{\theta}{\theta + k} - \frac{\theta k}{(\theta +k)^2} \right) =   \sum_{k=0}^{n-1} \frac{\theta^2}{(\theta + k)^2}, \]
which gives \eqref{eq:quadPsi}.
\end{proof}

\begin{lemma}\label{lem:borne}
For all $n\geq 1$,
\[ \sum_{1\leq j,k \leq n} \frac{1}{jk} \left\vert \Psi_n (j) \Psi_n (k) - \Psi_n (j+k) \mathds{1}_{j+k\leq n} \right\vert = \mathcal{O}_\theta (1).   \]
\end{lemma}
\begin{proof}
Discussing the sign of terms inside absolute values according to $\theta$, and using the previous lemma, we get
\begin{align*}
&\sum_{1\leq j,k \leq n} \frac{1}{jk} \left\vert \Psi_n (j) \Psi_n (k) - \Psi_n (j+k) \mathds{1}_{j+k\leq n} \right\vert  \\
	&\qquad = \sum_{\substack{1\leq j,k \leq n \\ j+k\leq n}} \frac{1}{jk} \left( \Psi_n (j) \Psi_n (k) - \Psi_n (j+k) \mathds{1}_{j+k\leq n} \right) \left( \mathds{1}_{\theta \geq 1} - \mathds{1}_{\theta <1}   \right) \\
	&\qquad \qquad + \sum_{\substack{1\leq j,k \leq n \\ j+k > n}} \frac{1}{jk} \Psi_n (j) \Psi_n (k) \left( \mathds{1}_{\theta \geq 1} + \mathds{1}_{\theta <1}   \right) \\
	&\qquad = \left( \mathds{1}_{\theta \geq 1} - \mathds{1}_{\theta <1}  \right)  \sum_{k=0}^{n-1} \frac{1}{(\theta + k)^2} + 2\mathds{1}_{\theta<1} \sum_{\substack{1\leq j,k \leq n \\ j+k > n}} \frac{1}{jk} \Psi_n (j) \Psi_n (k)  \\
	&\qquad = \mathcal{O}_\theta (1) + 2\mathds{1}_{\theta<1} \sum_{\substack{1\leq j,k \leq n \\ j+k > n}} \frac{1}{jk} \Psi_n (j) \Psi_n (k) .
\end{align*}
Assume $\theta <1$. It remains to show $\sum\limits_{\substack{1\leq j,k \leq n \\ j+k > n}} \frac{1}{jk} \Psi_n (j) \Psi_n (k) = \mathcal{O}_\theta (1)$. \\
To do this, we split this sum as the one for $j$ and $k$ between $\frac{n}{10}$ and $n$, plus the one for $j$ or $k$ between $1$ and $\frac{n}{10}$. Based on the observation that there exists a constant $C_\theta$ such that for all $n$ and $j\leq n$, $\Psi_n (j) \leq C_\theta \left(\frac{n}{n-j}\right)^{1-\theta}$, with the convention $\frac{1}{0^{1-\theta}}=1$, it comes
\begin{align*}
\sum\limits_{\substack{\frac{n}{10}\leq j,k \leq n \\ j+k > n}} \frac{1}{jk} \Psi_n (j) \Psi_n (k) &\ll_\theta n^{-2\theta} \sum_{\frac{n}{10} \leq j,k \leq  n} \left(\frac{1}{n-j}\right)^{1-\theta} \left(\frac{1}{n-k}\right)^{1-\theta} \\
	&=n^{-2\theta} \left( \sum_{0\leq j \leq \frac{9n}{10}} j^{-(1-\theta)} \right)^2 \ll_\theta 1
\end{align*}
and 
\begin{align*}
\sum\limits_{\substack{1\leq j \leq \frac{n}{10} \\ j+k > n}} \frac{1}{jk} \Psi_n (j) \Psi_n (k) &\ll_\theta n^{-\theta} \sum_{1\leq j \leq \frac{n}{10}} \frac{1}{j}  \sum_{k=n-j+1}^{n} \left(\frac{1}{n-k}\right)^{1-\theta} \\
	&\ll_\theta n^{-\theta} \sum_{1\leq j \leq \frac{n}{10}} \frac{1}{j} j^\theta \ll_\theta 1,
\end{align*}
which gives the claim.
\end{proof}

\section{Proof of Theorem 1}
\label{sec:3}

\subsection{Mean and variance}

\subsubsection{\texorpdfstring{Symmetric group $\mathfrak{S}_N$ }{Symmetric group}}

In order to compare both ensembles with each other, we recall here some known results on the counting function of eigenvalues for $\mathfrak{S}_N$ that one can find for example in \cite{wieand2000eigenvalue} (for the case $\theta=1$) or in \cite{ben2015fluctuations}: 
\begin{proposition}
For all $N\geq 1$, denoting $\omega_j:=(\{j\beta \} - \{ j\alpha \})$, 
\begin{equation}
\mathbb{E} (X_N^I ) = N(\beta - \alpha) - \theta \sum_{j=1}^N \frac{\omega_j}{j} \Psi_N (j)
\end{equation}
\begin{equation}
\mathrm{Var} (X_N^I ) = \theta \sum_{j=1}^N \frac{\omega_j^2}{j} \Psi_N (j)  +\theta^2\sum_{1\leq j,k \leq N} \frac{\omega_j \omega_k}{jk} (\Psi_N (j+k) \mathds{1}_{j+k\leq N} - \Psi_N (j) \Psi_N (k)) . 
\end{equation}
\end{proposition}

\begin{proposition}\label{prop:EspVarX}
There exists a real number $c_1=c_1(\alpha, \beta)$ and a positive real number $c_2=c_2(\alpha , \beta)$ such that
\begin{equation}
\mathbb{E} (X_N^I) \underset{N\to \infty}{=} N(\beta - \alpha) + c_1 \theta \log N + o (\log N)
\end{equation}
\begin{equation}
\mathrm{Var} (X_N^I) \underset{N\to \infty}{\sim} c_2 \theta \log N. 
\end{equation}
\end{proposition}

We give two significant examples of values taken by $c_2$:
\begin{itemize}
	\item If $\alpha$ and $\beta$ are irrational and linearly independent over $\mathbb{Q}$, then $c_2=\frac{1}{6}$.
	\item If $\beta$ is irrational and $\alpha = \frac{p}{q}$ with $p,q$ coprime numbers, then $c_2= \frac{1}{6} + \frac{1}{6q^2}$.
\end{itemize}

Details for the computation of the coefficient $c_2$ and more examples are given in \cite{wieand2000eigenvalue}. We complete its study in our Appendix.

\subsubsection{\texorpdfstring{Wreath product $S^1 \wr \mathfrak{S}_N$}{Wreath product}}

To begin with, we give a simple expression of $\widetilde{X}_N^I$ in function of the random variables $(a_{N,j})_{1\leq j \leq N}$ and $(T_{j,p})$, where the law of $T_{j,p}$ is the multiplicative convolution of $j$ independent copies of the uniform distribution on $S^1$, \emph{i.e} the uniform distribution on $S^1$ itself, and where we recall that $a_{N,j}$ denotes the number of $j$-cycles in $\sigma_N$. We have the following equalities in distribution:
\begin{equation}
\widetilde{X}_N^I = \sum_{j=1}^N \sum_{p=1}^{a_{N,j}} \sum_{w^j = T_{j,p}} \mathds{1}_{w\in I} = \sum_{j=1}^N \sum_{p=1}^{a_{N,j}} \sum_{w^j = \mathrm{e}^{2i\pi \phi_{j,p}} } \mathds{1}_{w\in I} 
\end{equation}
where the $(\phi_{j,p})$ are i.i.d random variables, uniformly distributed on $[0,1)$.\\
Recalling that $I = \left( \mathrm{e}^{2i\pi \alpha} , \mathrm{e}^{2i\pi \beta} \right]$, this can be reformulated this way:
\begin{align*}
\widetilde{X}_N^I &= \sum_{j=1}^N \sum_{p=1}^{a_{N,j}} \left(j (\beta - \alpha ) - \{j\beta - \phi_{j,p} \} + \{j\alpha - \phi_{j,p} \} \right) \\
	&= (\beta - \alpha ) \sum_{j=1}^N j a_{N,j} - \sum_{j=1}^N \sum_{p=1}^{a_{N,j}} \left( \{j\beta \} - \{j\alpha \} + \mathds{1}_{\phi_{j,p} > \{ j \beta \} } - \mathds{1}_{\phi_{j,p} > \{ j \alpha \} }   \right) \\
	&= N (\beta - \alpha ) - \sum_{j=1}^N a_{N,j} \left( \{j\beta \} - \{j\alpha \}\right)  - \sum_{j=1}^N \sum_{p=1}^{a_{N,j}} \left(  \mathds{1}_{\phi_{j,p} > \{ j \beta \} } - \mathds{1}_{\phi_{j,p} > \{ j \alpha \} }    \right).
\end{align*}

\begin{remark}
It can be noticed that in contrast to the classical ensemble of permutation matrices, if we include the lower endpoint and/or exclude the upper endpoint of the interval, then almost surely the value of the counting function on this interval remains the same. 
\end{remark}

\begin{proposition}
For all $N\geq 1$,
\begin{equation}
\mathbb{E} (\widetilde{X}_N^I ) = N(\beta - \alpha)
\end{equation}
\begin{equation}\label{eq:varXtilde}
\mathrm{Var} (\widetilde{X}_N^I ) = \theta \sum_{j=1}^N \frac{\Psi_N (j)}{j}  \{ j (\beta -\alpha) \}\left( 1 - \{ j(\beta-\alpha) \} \right). 
\end{equation}
\end{proposition}

\begin{remark}
For the classical ensemble of permutation matrices, we had a expectation which weakly depended (additional term in $\log N$) on the arithmetic nature of the endpoints of the interval. This is not the case here, and this phenomenon can be well understood since the modification operates uniform random shifts on the sets of eigenangles corresponding to each cycle. Regarding the variance, the effect of endpoints does not vanish since we still have some fractional parts in its expression. More specifically, we have an effect induced by the difference of the endpoints.
\end{remark}

\begin{proof}
First of all we consider the conditional expectation with respect to the random permutation $\sigma_N$. We have
\[\mathbb{E} [ \widetilde{X}_N^I \ | \ \sigma_N ] = N(\beta - \alpha)  + \sum_{j=1}^N \sum_{p=1}^N \mathbb{E} [ b_{j,p} \mathds{1}_{a_{N,j} \geq p } \ |  \ \sigma_N ]  \]
where $b_{j,p} := \{ j \alpha \} - \{ j \beta \} + \mathds{1}_{\phi_{j,p} > \{ j \alpha \} } - \mathds{1}_{\phi_{j,p} > \{ j \beta \} }$. \\
Since the information of $\sigma_N$ provides all the information of its cycle structure (\emph{i.e} the numbers of its cycles of the same sizes), it follows that for all $j,p$, 
\[\mathbb{E} [ b_{j,p} \mathds{1}_{a_{N,j} \geq p } \ |  \ \sigma_N ] = \mathds{1}_{a_{N,j} \geq p } \mathbb{E} [ b_{j,p}  \ |  \ \sigma_N ].\] Moreover the $b_{j,p}$ are independent of $\sigma_N$, hence \[\mathbb{E} [ b_{j,p}  \ |  \ \sigma_N ] = \mathbb{E} (b_{j,p}) = \{ j \alpha \} - \{ j \beta \} + \mathbb{P} (\phi_{j,p} > \{ j \beta \} ) - \mathbb{P} (\phi_{j,p} > \{ j \alpha \} )=0.\] Consequently all the terms in the double series are zeros. Finally,
\[\mathbb{E} ( \widetilde{X}_N^I )= \mathbb{E} [\mathbb{E} [ \widetilde{X}_N^I \ | \ \sigma_N ]] = \mathbb{E} [ N (\beta - \alpha ) ] = N (\beta - \alpha).\]

The computation of the variance is a little longer. Using the fact that the $b_{j,p}$ are centred and independent,
\begin{align*}
\mathrm{Var} (\widetilde{X}_N^I) &= \mathbb{E} ((\widetilde{X}_N^I - \mathbb{E}(\widetilde{X}_N^I))^2) \\
	&= \mathbb{E} \left( \left[ \sum_{j=1}^N \sum_{p=1}^N \mathbb{E} ( b_{j,p} \mathds{1}_{a_{N,j} \geq p } )  \right]^2 \right) \\
	&= \sum_{j=1}^N \mathbb{E} \left( \left( \sum_{p=1}^{a_{N,j}} b_{j,p} \right)^2 \right) + 2 \sum_{1\leq j < k \leq N} \sum_{p=1}^N \sum_{m=1}^N \mathbb{E} ( b_{j,p} b_{k,m} \mathds{1}_{a_{N,j} \geq p } \mathds{1}_{a_{N,k} \geq m }) \\
	&= \sum_{j=1}^N \left[ \sum_{p=1}^N \mathbb{E} (b_{j,p}^2 \mathds{1}_{a_{N,j} \geq p} ) + 2 \sum_{1\leq m < p \leq N} \mathbb{E} (b_{j,p} b_{j,m} \mathds{1}_{a_{N,j} \geq p} \mathds{1}_{a_{N,j} \geq m}) \right] \\
	&= \sum_{j=1}^N \sum_{p=1}^N \mathbb{E} (b_{j,p}^2) \mathbb{P} (a_{N,j} \geq p)  
\end{align*}
with for all $j,p$, 
\begin{align*}
\mathbb{E} (b_{j,p}^2) &= \mathrm{Var} ( \mathds{1}_{\phi_{j,p} > \{ j \alpha \} } - \mathds{1}_{\phi_{j,p} > \{ j \beta \} }) \\
	&= \mathbb{E} \left( \mathds{1}_{\phi_{j,p} > \{ j \alpha \} } + \mathds{1}_{\phi_{j,p} > \{ j \beta \} } - 2 \mathds{1}_{\phi_{j,p} > \{ j \alpha \} } \mathds{1}_{\phi_{j,p} > \{ j \beta \} } \right) - ( \{ j \beta \} - \{ j \alpha \})^2 \\
	&= 2 - \{ j \alpha \} - \{ j \beta \} - 2 \mathbb{P} ( \phi_{j,p} > \max (\{ j \alpha \}, \{ j \beta \}) ) -( \{ j \beta \} - \{ j \alpha \})^2  \\
	&= |\{ j\alpha \} - \{ j\beta \}| - (\{ j\alpha \} - \{ j\beta \})^2 .
\end{align*}
Hence,
\begin{align*}
\mathrm{Var} (\widetilde{X}_N^I) &= \sum_{j=1}^N \left[ |\{ j\alpha \} - \{ j\beta \}| - (\{ j\alpha \} - \{ j\beta \})^2 \right] \sum_{p=1}^N \mathbb{P} (a_{N,j} \geq p) \\
	&= \sum_{j=1}^N \mathbb{E} (a_{N,j}) \left[ |\{ j\alpha \} - \{ j\beta \}| - (\{ j\alpha \} - \{ j\beta \})^2 \right]
\end{align*}
and we know (see \cite{arratia2003logarithmic} page 96) that for all $1\leq j \leq N$, 
\begin{equation}
\mathbb{E} [a_{N,j}] = \frac{\theta}{j} \cdot \frac{N(N-1) \cdots (N-j+1) }{(\theta + N-1) \cdots (\theta + N-j)} = \frac{\theta}{j} \Psi_N (j).
\end{equation}
It remains to see that for all $j$, $|\{ j\alpha \} - \{ j\beta \}| - (\{ j\alpha \} - \{ j\beta \})^2 = \{ j(\beta - \alpha ) \} (1 - \{j (\beta - \alpha ) \})$. Indeed, it derives from the next lemma:
\begin{lemma}
Let $x,y$ be real numbers. Then for all $T\in \mathbb{R}$, \[  \left| \{x+T\} - \{y+T\} \right| (1 - \left| \{x+T\} - \{y+T\} \right| ) =  \left| \{x\} - \{y\} \right| (1 - \left| \{x\} - \{y\} \right| ). \]
\end{lemma}
\begin{proof}
Without loss of generality we can suppose $T \in [0,1)$ (since $\{x+T\} = \{x + \{T\} \}$). We notice first that \[\{x+T\} - \{y+T\} = \{x\} - \{y\} + \mathds{1}_{ \{y\} > 1 -T } - \mathds{1}_{ \{x\} > 1 -T }.  \]
By discussing the relative positions of $\{ x\}$ and $\{y\}$ with respect to $1-T$, the difference of indicator functions $\mathds{1}_{ \{y\} > 1 -T } - \mathds{1}_{ \{x\} > 1 -T }$ takes the value $-1,0$ or $1$. Thus it is easy to check that in all cases the equality holds.
\end{proof}
Applying this lemma with $x=j\beta$ and $y=-T=j\alpha$ for $j\in \mathbb{N}^*$, we deduce \eqref{eq:varXtilde}.
\end{proof}

The next proposition is conform to the intuition one could have as regards with the asymptotic of the variance. We make use of what we set up in preliminaries to prove it.

\begin{proposition}
\begin{equation}\label{eq:asymptVarXtilde}
\mathrm{Var} (\widetilde{X}_N^I) \underset{N\to \infty}{\sim} \ell \theta \log N 
\end{equation}
where \[\ell := \left\{\begin{array}{ll}
\frac{1}{6} & \text{if } \beta-\alpha \text{ is irrationnal}\\
\frac{1}{6} - \frac{1}{6q^2} & \text{if } \beta-\alpha=\frac{p}{q} \text{ with } p,q \text{ coprime integers}, q\geq 2.
\end{array}\right.\]
\end{proposition}
\begin{proof}
Let us define $\delta = \beta - \alpha$ and $w_j = \{j \delta \} (1 - \{j \delta \} )$, $j \in \mathbb{N}^*$, in such a way that
\[\mathrm{Var} (\widetilde{X}_N^I) =  \theta \sum_{j=1}^N \frac{\Psi_N (j)}{j} w_j.\] 
First, we notice that the sequence $(w_j)$ is non-negative and bounded (by $1$). Moreover, it is proven in \cite{wieand2000eigenvalue} that the limits $\ell_1 := \lim\limits_{N \to \infty} \frac{1}{N} \sum\limits_{j=1}^N  \{j\delta\}$ and $\ell_2 := \lim\limits_{N \to \infty} \frac{1}{N} \sum\limits_{j=1}^N  \{j\delta\}^2$ exist and are finite. Their respective explicit values depend on whether $\delta$ is rational or irrational. More precisely,
\begin{itemize}
	\item if $\delta$ is irrational, then $\ell_1 = \frac{1}{2}$ and $\ell_2 = \frac{1}{3}$, 
	\item if $\delta = \frac{p}{q}$ with $q\geq 1$ and $\mathrm{gcd}(p,q)=1$, then $\ell_1 = \frac{q(q-1)}{	2q^2}$ and $\ell_2 = \frac{(2q-1)q(q-1)}{6q^3}$.
\end{itemize}
Thus, the sequence $(w_j)$ converges $(C,1)$ to $\ell := \ell_1 -\ell_2>0$. Consequently, we can apply Proposition~\ref{prop:cesaro} on $w=(w_j)$ so that $\sigma_N^\theta (w) \underset{N\to \infty}{\longrightarrow} \ell$. \\
Finally, since $s_N := \frac{1}{N} \sum_{j=1}^N \Psi_N (j) w_j \underset{N\to \infty}{\sim} \frac{\sigma_N^\theta (w)}{\theta}$, then $s_N  \underset{N\to \infty}{\longrightarrow} \frac{\ell}{\theta} = :L >0$. It follows from Proposition~\ref{prop:log} that \[ \sum\limits_{j=1}^N \frac{\Psi_N (j)}{j} w_j \underset{N\to \infty}{\sim} L \theta \log N = \ell \log N \]
which implies \eqref{eq:asymptVarXtilde}.
The computation of $\ell_1$ and $\ell_2$ is detailed in \cite{wieand2000eigenvalue}.
\end{proof}

\subsection{Limiting normality for a finite number of fixed arcs}

We consider a finite number of fixed arcs $I_1 , \cdots , I_m$ on the unit circle, where $I_k := \left( e^{2i\pi \alpha_k} , e^{2i\pi \beta_k} \right]$. For $1\leq k \leq m$, we denote by $c_2^{(k)}$ and $\ell^{(k)}$ the respective constant numbers appearing in the asymptotic expressions of the variances of $X_N^{I_k}$ and $\widetilde{X}_N^{I_k}$. \\ 
In order to simplify notations, we also define for $1\leq k,l \leq m$,
\[
\omega_{j,k} := \{j\beta_k \} - \{ j\alpha_k \} \]
and 
\[H_{j,k,l} := \frac{1}{2} \left( \vert \{j\beta_k \} - \{ j\alpha_l \} \vert + \vert \{j\alpha_k \} - \{ j\beta_l \} \vert - \vert \{j\alpha_k \} - \{ j\alpha_l \} \vert   - \vert \{j\beta_k \} - \{ j\beta_l \} \vert    \right) - \omega_{j,k} \omega_{j,l}.\]

\begin{lemma}\label{lem:Hcesaro}
The sequences $(H_{j,k,l})_{j\geq 1}$ and $(\omega_{j,k} \omega_{j,l})_{j\geq 1}$ converge $(C,1)$.
\end{lemma}
\begin{proof}
One can notice that \[H_{j,k,l} = \frac{1}{2} ( h_j (\alpha_k - \beta_l )+ h_j (\beta_k - \alpha_l )  - h_j (\alpha_k - \alpha_l ) - h_j (\beta_k - \beta_l ))\] where $h_j (x) := \{jx \}(1 - \{jx\})$.
Moreover, it is clear that for all fixed real numbers $x$,
\begin{itemize}
	\item if $x = \frac{p}{q} \in \mathbb{Q}$, the sequence $(\{j x\})_{j\geq 1}$ is $q$-periodic.
	\item if $x\in \mathbb{R}\setminus \mathbb{Q}$, the sequence $(\{j x\})_{j\geq 1}$ is equidistributed on $[0,1]$.
\end{itemize}
Now, if a sequence is periodic then this sequence converges $(C,1)$. Furthermore, if a sequence is equidistributed on $[0,1]$ then all continuous functions on $[0,1]$ applied to it converge $(C,1)$ to the integral of these functions on $[0,1]$. \\
Thus, for all $x = \alpha_k - \beta_l , \beta_k - \alpha_l , \alpha_k - \alpha_l , \beta_k - \beta_l $, the sequence $(h_j(x))_{j\geq 1}$ converges $(C,1)$. \\
For $\omega_{j,k} \omega_{j,l}$ it is much more difficult since we have to deal with some products of type $\{jx\}\{jy\}$ for $x,y$ real numbers. Discussing the rationality of $x$ and $y$ and eventually their linearly dependence over $\mathbb{Q}$ in the case where they are both irrational, it can be shown that $(\{jx\}\{jy\})_{j\geq 1}$ converges $(C,1)$ to explicit limits, and thus $(\omega_{j,k} \omega_{j,l})_{j\geq 1}$ converges $(C,1)$. See \cite{wieand2000eigenvalue} and our Appendix for details.  
\end{proof}

\everymath{\displaystyle}
\begin{theorem}\label{th:TCLfix}
Let $Y_N^{I_k} := \frac{X_N^{I_k} - \mathbb{E} ( X_N^{I_k}) }{(c_2^{(k)} \theta \log N)^{1/2}}$, and let $Z=(Z_1, \cdots , Z_m) \sim \mathcal{N}(0, D)$ where $D$ is the covariance matrix defined for all $k,l$ by \[D_{k,l} := \frac{1}{\sqrt{c_2^{(k)} c_2^{(l)}}} \lim\limits_{N\to +\infty} \frac{1}{N} \sum\limits_{j=1}^N  \omega_{j,k} \omega_{j,l}. \] Then
\[(Y_N^{I_1} , \cdots , Y_N^{I_m}) \overset{\text{d}}{\longrightarrow} (Z_1, \cdots , Z_m).\]
\end{theorem}
\everymath{}
\begin{proof}
Let $t:=(t_1, \cdots , t_m) \in \mathbb{R}^m$. The theorem will be proven if we show that $t_1 Y_N^{I_1} + \cdots + t_m Y_N^{I_m} \overset{\text{d}}{\longrightarrow} t_1 Z_1+ \cdots + t_m Z_m$.
The main idea of the proof is to replace in the expression $t_1 Y_N^{I_1} + \cdots + t_m Y_N^{I_m}$ the $a_{N,j}$ by independent Poisson random variables $W_j$ given by the Feller Coupling, show that the difference converges in probability to zero, and then use the new expression to show the convergence in distribution, and finally conclude with Slutsky's theorem. Note that this scheme is very typical when one deals with random permutations, since the approximation of cycle counts by Poisson random variables is natural. \\
Let 
\[V_{N,j}^{(t)} := \left( \frac{\theta}{j} - W_j \right) \sum_{k=1}^m \frac{t_k}{(\theta c_2^{(k)} \log N)^{1/2}} \omega_{j,k}.\]
Since the $W_j$ are independent, then $(V_{N,j}^{(t)})_{N\geq 1, 1\leq j \leq N }$ is a triangular array of independent random variables. Let $T_N^{(t)}:= \sum_{j=1}^N V_{N,j}^{(t)}$. Then
\begin{align*}
t_1 Y_N^{I_1} + \cdots + t_m Y_N^{I_m} - T_N^{(t)} &= \sum_{k=1}^m \frac{t_k}{(\theta c_2^{(k)} \log N)^{1/2}} \left[ X_N^{I_k} - \mathbb{E} (X_N^{I_k}) - \sum_{j=1}^N \left( \frac{\theta}{j} - W_j \right) \omega_{j,k} \right] \\
	&= \sum_{k=1}^m \frac{t_k}{(\theta c_2^{(k)} \log N)^{1/2}} \sum_{j=1}^N \left[ \frac{\theta}{j} \Psi_N (j) - \frac{\theta}{j} + W_j - a_{N,j} \right] \omega_{j,k}.
\end{align*}
We have
\begin{align*}
\vert t_1 Y_N^{I_1} + \cdots + t_m Y_N^{I_m} - T_N^{(t)} \vert &\leq \sum_{k=1}^m \frac{\vert t_k \vert }{(\theta c_2^{(k)} \log N )^{1/2}} \left[  \sum_{j=1}^N \frac{\theta}{j} \vert \Psi_N (j) - 1 \vert  + \sum_{j=1}^N \vert W_j - a_{N,j}  \vert \right] 
\end{align*}
with for all $j$, $\vert \Psi_N (j) - 1 \vert = (\Psi_N (j) - 1 ) (\mathds{1}_{0<\theta <1} - \mathds{1}_{\theta \geq 1})$, thus using Lemma~\ref{lem:EqPsi}, 
\[ \sum_{j=1}^N \frac{\theta}{j} \vert \Psi_N (j) - 1 \vert  = \theta \vert 1-\theta \vert \sum_{j=1}^N \frac{1}{j (\theta + j-1)} = \mathcal{O}_\theta (1).\] 
In addition, by Lemma~\ref{lem:majPoisson}, $ \mathbb{E} \left( \sum_{j=1}^N \vert W_j - a_{N,j} \vert \right) =  \mathcal{O}_\theta (1)$. \\ 
Let $\varepsilon >0$. Then, using Markov's inequality,
\[\mathbb{P} ( \vert t_1 Y_N^{I_1} + \cdots + t_m Y_N^{I_m} - T_N^{(t)} \vert > \varepsilon ) \leq \frac{1}{\varepsilon (\log N )^{1/2}} \sum_{k=1}^m  \frac{\vert t_k \vert}{(\theta c_2^{(k)})^{1/2}} \mathcal{O}_\theta (1) \underset{N\to \infty}{\longrightarrow}  0.\]
Now, we want to show that $T_N^{(t)}$ is asymptotically normal. For this purpose, we will check the condition of Lindeberg-Feller on $(V_{N,j}^{(t)})$:
\begin{equation}\label{eq:condLindeberg}
\lim_{N\to\infty} \sum_{j=1}^N \mathbb{E} \left( \left( V_{N,j}^{(t)}\right)^2 \mathds{1}_{\vert  V_{N,j}^{(t)} \vert > \varepsilon} \right) = 0.
\end{equation}
We have the bound
\[\vert V_{N,j}^{(t)} \vert \leq \sum_{k=1}^m \frac{\vert t_k \vert }{(\theta c_2^{(k)} \log N)^{1/2}} \left\vert W_j - \frac{\theta}{j} \right\vert \leq \frac{C}{(\log N)^{1/2}}  \left\vert W_j - \frac{\theta}{j} \right\vert\]
where $C:= \sum\limits_{k=1}^m \frac{\vert t_k \vert }{(\theta c_2^{(k)})^{1/2}}$, we deduce that for any $\varepsilon >0$, so that
\begin{equation}\label{eq:ineqLindeberg1}
\sum_{j=1}^N \mathbb{E} \left( \left( V_{N,j}^{(t)}\right)^2 \mathds{1}_{\vert V_{N,j}^{(t)}  \vert > \varepsilon } \right) \leq \frac{C^2}{\log N} \sum_{j=1}^N \mathbb{E} \left( \left( W_j-\frac{\theta}{j}\right)^2 \mathds{1}_{\left\vert W_j - \frac{\theta}{j} \right\vert > \frac{\varepsilon}{C}(\log N)^{1/2} } \right). 
\end{equation}
We could try a fourth moment bound (Lyapunov condition) but it is not sufficient since $\mathbb{E} \left( \left( W_j - \frac{\theta}{j} \right)^4 \right) = \mathcal{O}(\theta /j)$. Now, we observe that for all $j \geq \lceil \theta \rceil$, 
\begin{align*}
\mathbb{E} \left( \left( W_j-\frac{\theta}{j}\right)^2  \mathds{1}_{\left\vert W_j - \frac{\theta}{j} \right\vert > \frac{\varepsilon}{C}(\log N)^{1/2} } \right) &= \frac{\theta^2}{j^2} \mathds{1}_{\frac{\theta}{j}>\frac{\varepsilon}{C}(\log N)^{1/2} } \mathrm{e}^{-\theta/j} \\
	&\qquad + \left(1-\frac{\theta}{j}\right)^2 \mathds{1}_{\left\vert 1 - \frac{\theta}{j} \right\vert > \frac{\varepsilon}{C}(\log N)^{1/2} }  \frac{\theta}{j} \mathrm{e}^{-\theta/j} \\
	&\qquad + \sum_{k=2}^{+\infty} \left( k-\frac{\theta}{j}\right)^2  \mathds{1}_{\left\vert k - \frac{\theta}{j} \right\vert > \frac{\varepsilon}{C}(\log N)^{1/2} } \frac{\theta^k}{j^k k!} \mathrm{e}^{-\theta/j} \\
	&\leq \frac{\theta^2}{j^2} +  \frac{\theta}{j} \mathrm{e}^{-\theta/j} \mathds{1}_{1>\frac{\varepsilon}{C}(\log N)^{1/2} } + \sum_{k=2}^{+\infty} k^2 \frac{\theta^k}{j^k k!} \mathrm{e}^{-\theta/j} \\
	&\leq \frac{\theta^2}{j^2}  +  \mathds{1}_{1>\frac{\varepsilon}{C}(\log N)^{1/2} } +  \sum_{k=2}^{+\infty} 2k(k-1) \frac{\theta^k}{j^k k!} \mathrm{e}^{-\theta/j} \\
	&= \frac{3\theta^2}{j^2} + \mathds{1}_{1>\frac{\varepsilon}{C}(\log N)^{1/2} },
\end{align*}
hence \[\sum_{j=1}^N\mathbb{E} \left( \left( W_j-\frac{\theta}{j}\right)^2  \mathds{1}_{\left\vert W_j - \frac{\theta}{j} \right\vert > \frac{\varepsilon}{C}(\log N)^{1/2} } \right)  \leq N \mathds{1}_{1>\frac{\varepsilon}{C}(\log N)^{1/2} }  + \mathcal{O}_\theta (1) \]
which, jointly to the bound \eqref{eq:ineqLindeberg1} allows to conclude that \eqref{eq:condLindeberg} is verified. 
Consequently the Lindeberg-Feller theorem applies and gives \[ T_N^{(t)} \overset{\text{d}}{\longrightarrow} \mathcal{N}(0,\sigma^2 )  \]
where
\begin{align*}
\sigma^2 &:= \lim_{N\to +\infty } \sum_{j=1}^N \mathbb{E} \left(\left( V_{N,j}^{(t)}\right)^2\right) \\
	&= \lim_{N\to +\infty } \sum_{j=1}^N \sum_{k,l = 1}^m \frac{t_k t_l}{\theta (c_2^{(k)} c_2^{(l)})^{1/2} \log N}  \omega_{j,k} \omega_{j,l} \mathrm{Var} (W_j) \\
	&= \sum_{k,l = 1}^m \frac{t_k t_l}{(c_2^{(k)} c_2^{(l)})^{1/2}} \lim_{N\to +\infty } \frac{1}{\log N} \sum_{j=1}^N  \frac{\omega_{j,k} \omega_{j,l}}{j}. 
\end{align*}
Finally, since for all $k,l$, $(\omega_{j,k} \omega_{j,l})_{j\geq 1}$ is bounded and converges $(C,1)$ (Lemma~\ref{lem:Hcesaro}), then it follows from Proposition~\ref{prop:log} (taking $\theta=1$) that
\[ \sigma^2 = \lim_{N\to +\infty} \frac{1}{N} \sum_{j=1}^N \sum_{k,l=1}^m \frac{t_k t_l}{(c_2^{(k)} c_2^{(l)})^{1/2}} \omega_{j,k} \omega_{j,l}.\]
Slutsky's theorem ends the proof.
\end{proof}

\everymath{\displaystyle}
\begin{theorem}\label{th:TCLWreath}
Let $\widetilde{Y}_N^{I_k} := \frac{\widetilde{X}_N^{I_k} - \mathbb{E} ( \widetilde{X}_N^{I_k}) }{(\ell^{(k)} \theta \log N)^{1/2}}$, and let $\widetilde{Z}=(\widetilde{Z}_1, \cdots , \widetilde{Z}_m) \sim \mathcal{N}(0, \widetilde{D})$ where $\widetilde{D}$ is the covariance matrix defined for all $k,l$ by $\widetilde{D}_{k,l} := \frac{1}{\sqrt{\ell^{(k)} \ell^{(l)}}} \lim\limits_{N\to +\infty} \frac{1}{N} \sum\limits_{j=1}^N H_{j,k,l}$. Then
\[(\widetilde{Y}_N^{I_1} , \cdots , \widetilde{Y}_N^{I_m}) \overset{\text{d}}{\longrightarrow} (\widetilde{Z}_1, \cdots , \widetilde{Z}_m).\]
\end{theorem}
\everymath{}
\begin{proof}
Let $t:=(t_1, \cdots , t_m) \in \mathbb{R}^m$. Again, it suffices to show that $t_1 \widetilde{Y}_N^{I_1} + \cdots + t_m \widetilde{Y}_N^{I_m} \overset{\text{d}}{\longrightarrow} t_1 \widetilde{Z}_1+ \cdots + t_m \widetilde{Z}_m$. In order to shorten the following expressions, let us define for $j,p \in \mathbb{N}^*$ and $1\leq k \leq m$,
\[b_{j,p,k} := \mathds{1}_{\phi_{j,p} > \{j\alpha_k \}} - \mathds{1}_{\phi_{j,p} > \{j\beta_k \}} - \omega_{j,k}\]
and \[B_{j,k}:=\sum_{p=1}^{W_j} b_{j,p,k}.\]
Let \[V_{N,j}^{(t)} := \sum_{k=1}^m \frac{t_k}{(\theta \ell^{(k)} \log N)^{1/2}} B_{j,k} .\]
Since the $W_j$ and the $\phi_{j,p}$ are independent, then $(V_{N,j}^{(t)})_{N\geq 1, 1\leq j \leq N }$ is a triangular array of independent random variables.
Let $T_N^{(t)}:= \sum_{j=1}^N V_{N,j}^{(t)}$. Then
\begin{align*}
t_1 \widetilde{Y}_N^{I_1} + \cdots + t_m \widetilde{Y}_N^{I_m} - T_N^{(t)} &= \sum_{k=1}^m \frac{t_k}{(\theta \ell^{(k)} \log N)^{1/2}} \left[ \widetilde{X}_N^{I_k} - \mathbb{E} (\widetilde{X}_N^{I_k}) - \sum_{j=1}^N B_{j,k} \right] \\
	&=  \sum_{k=1}^m \frac{t_k}{(\theta \ell^{(k)} \log N)^{1/2}} \sum_{j=1}^N \left[ \sum_{p=1}^{a_{N,j}} b_{j,p,k} - B_{j,k} \right].
\end{align*}
This quantity converges in probability to $0$. Indeed, let $\varepsilon >0$. Using Markov's inequality,
\begin{align*}
\mathbb{P} ( \vert t_1 \widetilde{Y}_N^{I_1} + \cdots + t_m \widetilde{Y}_N^{I_m} - T_N^{(t)} \vert > \varepsilon ) &\leq \frac{1}{\varepsilon} \mathbb{E}  ( \vert t_1 \widetilde{Y}_N^{I_1} + \cdots + t_m \widetilde{Y}_N^{I_m} - T_N^{(t)} \vert)  \\
	&\leq \frac{1}{ \sqrt{ \log N }} \sum_{k=1}^m \frac{\vert t_k \vert }{\varepsilon (\theta \ell^{(k)} )^{1/2}} \mathbb{E} \left( \sum_{j=1}^N \left\vert \sum_{p=1}^{a_{N,j}} b_{j,p,k} - B_{j,k} \right\vert \right)
\end{align*}
with \begin{align*}
\mathbb{E} \left( \sum_{j=1}^N \left\vert \sum_{p=1}^{a_{N,j}} b_{j,p,k} - B_{j,k} \right\vert \right) &\leq \mathbb{E} \left( \sum_{j=1}^N  \sum_{p=1}^{\vert a_{N,j} - W_j \vert} \vert b_{j,p,k} \vert \right) \\
	&\leq \mathbb{E} \left( \sum_{j=1}^N \vert a_{N,j} - W_j \vert \right) \\
	&= \mathcal{O}_\theta (1)
\end{align*}
applying Lemma~\ref{lem:majPoisson}.
Hence \[\mathbb{P} ( \vert t_1 \widetilde{Y}_N^{I_1} + \cdots + t_m \widetilde{Y}_N^{I_m} - T_N^{(t)} \vert > \varepsilon )  \underset{N\to \infty}{\longrightarrow}  0.\]
Now, we want to show that $T_N^{(t)}$ is asymptotically normal. For this purpose, we check the condition of Lindeberg-Feller on $(V_{N,j}^{(t)})$. Noticing that 
\[\vert V_{N,j}^{(t)} \vert \leq \sum_{k=1}^m \frac{\vert t_k \vert }{(\theta \ell^{(k)} \log N)^{1/2}} \vert B_{j,k} \vert \leq \frac{C}{(\log N)^{1/2}}  W_j\]
where $C:= \sum\limits_{k=1}^m \frac{\vert t_k \vert }{(\theta \ell^{(k)})^{1/2}}$, we deduce that for any $\varepsilon >0$,
\begin{equation}\label{eq:ineqLindeberg2}
\sum_{j=1}^N \mathbb{E} \left( \left( V_{N,j}^{(t)}\right)^2 \mathds{1}_{\vert V_{N,j}^{(t)}  \vert > \varepsilon } \right) \leq \frac{C^2}{\log N} \sum_{j=1}^N \mathbb{E} \left( W_j^2 \mathds{1}_{W_j > \frac{\varepsilon}{C}(\log N)^{1/2} } \right).
\end{equation}
From this, we observe that for all $j$, 
\begin{align*}
\mathbb{E} \left( W_j^2 \mathds{1}_{W_j > \frac{\varepsilon}{C}(\log N)^{1/2} } \right) &= \sum_{k=1}^{+\infty} k^2 \mathds{1}_{k>\frac{\varepsilon}{C}(\log N)^{1/2} } \mathbb{P} (W_j=k) \\
	&= \frac{\theta}{j} \mathrm{e}^{-\theta/j} \mathds{1}_{1>\frac{\varepsilon}{C}(\log N)^{1/2} } + \mathcal{O}_\theta \left( \frac{1}{j^2} \right)
\end{align*}
hence \[\sum_{j=1}^N \mathbb{E} \left( W_j^2 \mathds{1}_{W_j > \frac{\varepsilon}{C}(\log N)^{1/2} } \right) \leq \theta \left( \sum_{j=1}^N \frac{1}{j} \right)   \mathds{1}_{1>\frac{\varepsilon}{C}(\log N)^{1/2} }  + \mathcal{O}_\theta (1) \]
which, jointly to the bound \eqref{eq:ineqLindeberg2} allows to conclude that \eqref{eq:condLindeberg} is verified. 
Consequently the Lindeberg-Feller theorem applies and gives 
\[ T_N^{(t)} \overset{\text{d}}{\longrightarrow} \mathcal{N}(0,\sigma^2 )  \]
where
\begin{align*}
\sigma^2 &:= \lim_{N\to +\infty } \sum_{j=1}^N \mathbb{E} \left(\left( V_{N,j}^{(t)}\right)^2\right) \\
	&= \lim_{N\to +\infty } \sum_{j=1}^N \sum_{k,l = 1}^m \frac{t_k t_l}{\theta (\ell^{(k)} \ell^{(l)})^{1/2} \log N} \mathbb{E} (B_{j,k} B_{j,l})
\end{align*}
with for all $j,k,l$,
\begin{align*}
\mathbb{E} (B_{j,k} B_{j,l}) &= \mathbb{E} \left( \sum_{p,q=1}^{W_j} b_{j,p,k} b_{j,q,l} \right) \\
	&= \sum_{p=1}^{+\infty} \mathbb{E} \left(  b_{j,p,k} b_{j,p,l} \right) \mathbb{P} (W_j \geq p) + 2 \sum_{\substack{p,q=1 \\ p<q}}^{+\infty} \mathbb{E} ( b_{j,p,k} )  \mathbb{E} ( b_{j,q,l} ) \mathbb{P} (W_j\geq q) \\
	&=\sum_{p=1}^{+\infty} [ \mathbb{E} ( (\mathds{1}_{\phi_{j,p} > \{j\alpha_k \}} - \mathds{1}_{\phi_{j,p} > \{j\beta_k \}}) (\mathds{1}_{\phi_{j,p} > \{j\alpha_l \}} - \mathds{1}_{\phi_{j,p} > \{j\beta_l \}}) ) \\
	&\qquad \qquad - \omega_{j,k} \omega_{j,l} ] \mathbb{P} (W_j \geq p) \\
	&= \mathbb{E} (W_j) \left[ \int_0^1 (\mathds{1}_{x > \{j\alpha_k \}} - \mathds{1}_{x > \{j\beta_k \}}) (\mathds{1}_{x > \{j\alpha_l \}} - \mathds{1}_{x > \{j\beta_l \}}) \mathrm{d}x  - \omega_{j,k} \omega_{j,l} \right] \\
	&=\frac{\theta}{j} \left[ \frac{1}{2}  ( \vert \{j\beta_k \} - \{ j\alpha_l \} \vert + \vert \{j\alpha_k \} - \{ j\beta_l \} \vert   \right. \\
	&\qquad \qquad \left.  - \vert \{j\alpha_k \} - \{ j\alpha_l \} \vert   - \vert \{j\beta_k \} - \{ j\beta_l \} \vert    ) -  \omega_{j,k} \omega_{j,l} \right] \\
	&= \frac{\theta}{j} H_{j,k,l}.
\end{align*}
Since for all $k,l$, the sequence $( H_{j,k,l} )_{j\geq 1}$ is bounded (by $1$) and converges $(C,1)$ to a finite limit (Lemma~\ref{lem:Hcesaro}), then it follows from Proposition~\ref{prop:log} (taking $\theta=1$) that
\[ \sigma^2 = \lim_{N\to +\infty} \frac{1}{N} \sum_{j=1}^N \sum_{k,l=1}^m \frac{t_k t_l}{(\ell^{(k)}\ell^{(l)})^{1/2}} H_{j,k,l}.\]
Slutsky's theorem ends the proof.
\end{proof}

Now, we can give a significant particular case of the two previous theorems: 
\begin{corollary}
With the same notation, let assume the $\alpha_k$ and the $\beta_l$, $1\leq k,l \leq m$, to be irrational numbers which are linearly independent over $\mathbb{Q}$. Then $D$ and $\widetilde{D}$ are the identity matrix.
\end{corollary}
Indeed, this corollary is a significant version since the additional condition is almost surely satisfied if the endpoints of the intervals are uniformly sampled on the unit circle.

\begin{proof}
Let $1\leq k,l \leq m$. Under the assumption, it suffices to notice that
\begin{itemize}
	\item If $k\neq l$, 
\[\lim_{N \to \infty } \frac{1}{N} \sum_{j=1}^N \omega_{j,k} \omega_{j,l} = \int_0^1  \int_0^1  \int_0^1  \int_0^1  (x-y) (z-t) \mathrm{d}x \mathrm{d}y \mathrm{d}z \mathrm{d}t = 0, \]
and
\[\lim_{N \to \infty } \frac{1}{N} \sum_{j=1}^N H_{j,k,l} = 0\]
thus $D_{k,l}=0$ and $\widetilde{D}_{k,l}=0$.
	\item If $k= l$, 
\[\lim_{N \to \infty } \frac{1}{N} \sum_{j=1}^N H_{j,k,l} + \omega_{j,k} \omega_{j,l}= \lim_{N\to \infty} \frac{1}{N} \sum_{j=1}^N \vert \omega_{j,k} \vert = \int_0^1 \int_0^1 \vert x-y \vert \mathrm{d}x \mathrm{d}y = \frac{1}{3}\]
and 
\[\lim_{N \to \infty } \frac{1}{N} \sum_{j=1}^N \omega_{j,k}^2 = \int_0^1 \int_0^1 (x-y)^2 \mathrm{d}x \mathrm{d}y = \frac{1}{6},\]
thus $D_{k,l} = 6 \times \frac{1}{6} = 1$ and $\widetilde{D}_{k,l}=6\left( \frac{1}{3} - \frac{1}{6}\right)=1$.
\end{itemize}
\end{proof}

\section{Proof of Theorem 2}
\label{sec:4} 

\subsection{Variance}

Here, we assume that the interval $I$ shrinks as $N$ tends to infinity. For $N\geq 1$ we define $I_N:= \left(\mathrm{e}^{2i\pi \alpha_N} , \mathrm{e}^{2i\pi \beta_N} \right]$, and $\delta_N :=\beta_N - \alpha_N \in (0,1]$.

\begin{proposition}
Suppose that the sequence $(\delta_N)$ satisfies $\left\{\begin{array}{l} 
\delta_N \underset{N\to \infty}{\longrightarrow} 0 \\
N \delta_N \underset{N\to \infty}{\longrightarrow} + \infty .
\end{array} \right.$
Then
\begin{equation}\label{eq:asymptVarXtildeMeso}
\mathrm{Var} (\widetilde{X}_N^{I_N}) \underset{N\to\infty}{\sim}  \frac{\theta}{6} \log(N\delta_N).
\end{equation}
\end{proposition}
\begin{proof}
First, with respect to what we have stated before, 
\[\mathrm{Var} (\widetilde{X}_N^{I_N}) =  \theta \sum_{j=1}^N \frac{\Psi_N (j) \chi_N (j)}{j}  \]
where $\chi_N (j) := \left\{j \delta_N \right\} \left(1 - \left\{j \delta_N \right\} \right)$.
We are going to study the particular case $\theta = 1$, then we will see that the general case $\theta>0$ can be quickly deduced from it.

\subsubsection*{\underline{Particular case $\theta =1$}:}

Since the function $t\mapsto \{ t \}$ is of period $1$, then $t \mapsto \chi_N (t)$ is periodic of period $1/\delta_N$. Thus, an idea may be to group the terms in the sum according to $1/\delta_N$. We begin with separating the points of the first period from those of the others:
\[\mathrm{Var} (\widetilde{X}_N^{I_N ,\theta=1}) = \sum_{j=1}^N \frac{\chi_N (j)}{j} = \sum_{j=1}^{\lceil 1/\delta_N \rceil -1} \frac{\chi_N (j)}{j} + \sum_{j=\lceil 1/\delta_N \rceil}^N \frac{\chi_N (j)}{j} \]
where 
\begin{align}\label{eq:varcase1part1}
\begin{split}
\sum\limits_{j=1}^{\lceil 1/\delta_N \rceil -1} \frac{\chi_N (j)}{j} &= \sum\limits_{j=1}^{\lceil 1/\delta_N \rceil -1} \frac{1}{j} j\delta_N \left(1-j\delta_N \right)  \\
	&= \delta_N \sum\limits_{j=1}^{\lceil 1/\delta_N \rceil -1}  \left(1-j\delta_N \right) \underset{N\to \infty}{\longrightarrow} \int_0^1 (1-x)\mathrm{d}x = \frac{1}{2}.
\end{split}
\end{align}
We apply a summation by parts on the second sum (Abel transformation):
\begin{align*}
\sum_{j=\lceil 1/\delta_N \rceil}^N \frac{\chi_N (j)}{j} &= \frac{1}{N} \sum_{j=\lceil 1/\delta_N \rceil}^N \chi_N (j) - \sum_{j=\lceil 1/\delta_N \rceil}^{N-1} \left( \frac{1}{j+1} - \frac{1}{j}  \right) \sum_{k=\lceil 1/\delta_N \rceil}^j \chi_N (k) \\
	&=\frac{1}{N} \sum_{j=\lceil 1/\delta_N \rceil}^N \chi_N (j) + \sum_{j=\lceil 1/\delta_N \rceil}^{N-1}  \frac{1}{j(j+1)} \sum_{k=\lceil 1/\delta_N \rceil}^j \chi_N (k)
\end{align*}
Let $j \geq \lceil 1/\delta_N \rceil $, and define 
\[A_j:= \frac{1}{j} \sum_{p=1}^{\left\lfloor j\delta_N \right\rfloor -1} \sum_{k=\lceil p/\delta_N \rceil}^{\lceil (p+1)/\delta_N\rceil -1} \chi_N (k) , \qquad B_j:= \frac{1}{j} \sum_{k=\left\lceil \left\lfloor j\delta_N \right\rfloor /\delta_N \right\rceil}^{j} \chi_N (k)   \]
in such a way that $ A_j+B_j = \frac{1}{j}  \sum\limits_{k=\lceil 1/\delta_N \rceil}^j \chi_N (k) $. On the one hand, we observe that
\[\vert B_j \vert \leq \frac{1}{j} \left(j-\left\lceil \left\lfloor j\delta_N \right\rfloor /\delta_N \right\rceil +1\right) \leq \frac{1+ 1/\delta_N}{j} \]
so
\begin{align*}
\left\vert \sum_{j=\lceil 1/\delta_N \rceil}^{N-1} \frac{B_j}{j+1}   \right\vert &\leq (1 + 1/\delta_N)\sum_{j=\lceil 1/\delta_N \rceil}^{N-1} \frac{1}{j(j+1)} = (1+ 1/\delta_N) \left( \frac{1}{\lceil 1/\delta_N \rceil} - \frac{1}{N} \right) \\
 	&= \mathcal{O}(1). 
\end{align*}
On the other hand,
\[A_j= \frac{1}{j} \sum_{p=1}^{\left\lfloor j \delta_N \right\rfloor -1} \sum_{q=0}^{\lceil (p+1)/\delta_N \rceil -1-\lceil p/\delta_N \rceil} \chi_N (q+\lceil p/\delta_N \rceil) \]
where 
\begin{align*}
\lceil (p+1)/\delta_N \rceil -1-\lceil p/\delta_N \rceil &=1/\delta_N + \{-(p+1)/\delta_N\} - \{-p/\delta_N\}-1 \\
	 &= \lfloor 1/\delta_N \rfloor - \mathds{1}_{\{-p/\delta_N \} \geq \{ - (p+1)/\delta_N\} }
\end{align*}
 and \[\chi_N (q+\lceil p/\delta_N \rceil) = \chi_N (q+ \{-p/\delta_N\}).\] Since for all $q \in [\![0 ,  \lceil (p+1)/\delta_N \rceil -1-\lceil p/\delta_N \rceil ]\!] $, 
\[ 0\leq \delta_N (q+ \{-p/\delta_N\}) \leq \delta_N( 1/\delta_N + \{-(p+1)/\delta_N \} -1) <1\]
then $\left\{ \delta_N (q+ \{-p/\delta_N\}) \right\} = \delta_N (q+ \{-p/\delta_N\})$, and 
\[A_j =\frac{1}{j} \sum_{p=1}^{\left\lfloor j\delta_N \right\rfloor -1} \  \sum_{q=0}^{\lfloor 1/\delta_N \rfloor - \mathds{1}_{\{-p/\delta_N \} \geq \{ - (p+1)/\delta_N \} }} \delta_N (q+ \{-p/\delta_N\})  (1 - \delta_N(q+ \{-p/\delta_N\})). \]
Let us introduce \[\widetilde{A}_j := \frac{1}{j} \left( \left\lfloor j\delta_N \right\rfloor -1   \right) \sum_{q=1}^{\lfloor 1/\delta_N \rfloor} q\delta_N ( 1- q\delta_N).\]
We have for all $j\geq \lceil 1/\delta_N \rceil$, 
\begin{align*}
\vert A_j - \widetilde{A}_j \vert &=  \left\vert \frac{1}{j} \sum_{p=1}^{\left\lfloor j\delta_N \right\rfloor -1} \  \sum_{q=0}^{\lfloor 1/\delta_N \rfloor - \mathds{1}_{\{-p/\delta_N \} \geq \{ - (p+1)/\delta_N \} }} \delta_N\{-p/\delta_N\} \left(1 - 2q \delta_N- \{-p/\delta_N\}\delta_N \right) \right. \\
 	&\left. \qquad - \mathds{1}_{\{-p/\delta_N\} \geq \{-(p+1)/\delta_N\}}  \delta_N \lfloor 1/\delta_N \rfloor  ( 1-  \delta_N \lfloor 1/\delta_N \rfloor ) \right\vert \\
	&\leq  \frac{1}{j} \sum_{p=1}^{\left\lfloor j\delta_N \right\rfloor -1} \  \sum_{q=0}^{\lfloor 1/\delta_N \rfloor - \mathds{1}_{\{-p/\delta_N \} \geq \{ - (p+1)/\delta_N \} }} \delta_N \{-p/\delta_N\} \left\vert 1 - 2q\delta_N -  \{-p/\delta_N\}\delta_N \right\vert \\
	&\qquad +  \mathds{1}_{\{-p/\delta_N\} \geq \{-(p+1)/\delta_N \}}  \delta_N \lfloor 1/\delta_N \rfloor ( 1- \delta_N \lfloor 1/\delta_N \rfloor )   \\
	&\leq \frac{1}{j} \left( \left\lfloor j\delta_N \right\rfloor -1 \right) \left[ 2\delta_N (1 + 1/\delta_N)  +  \delta_N \right] \leq 2\delta_N^2 (1/\delta_N + 2)
\end{align*}
hence
\begin{align*}
\left \vert \sum_{j=\lceil 1/\delta_N \rceil}^{N-1} \frac{A_j}{j+1} -  \sum_{j=\lceil 1/\delta_N \rceil}^{N-1} \frac{\widetilde{A}_j}{j+1}  \right\vert &\leq 2\delta_N^2 (1/\delta_N + 2) \sum_{j=\lceil 1/\delta_N \rceil}^{N-1}  \frac{1}{j+1} \\
	&\underset{N\to\infty}{\sim} 2\delta_N \log (N\delta_N) = o (\log ( N\delta_N )).
\end{align*}
Moreover, \[\sum_{j=\lceil 1/\delta_N \rceil}^{N-1} \frac{\widetilde{A}_j}{j+1} = \delta_N  \sum_{q=1}^{\lfloor 1/\delta_N \rfloor} q\delta_N \left( 1- q\delta_N \right)  \sum_{j=\lceil 1/\delta_N \rceil}^{N-1}  \left(  \frac{1}{j+1} - \frac{(1+\{j \delta_N\})/\delta_N }{j(j+1)}  \right)   \]
with \[\delta_N  \sum_{q=1}^{\lfloor 1/\delta_N \rfloor} q\delta_N \left( 1- q\delta_N \right) \underset{N\to\infty}{\longrightarrow} \int_0^1 x(1-x)\mathrm{d}x = \frac{1}{6}  \]
and \[\sum_{j=\lceil 1/\delta_N \rceil}^{N-1}  \left(  \frac{1}{j+1} - \frac{(1+\{j \delta_N\})/\delta_N }{j(j+1)}  \right)   \underset{N\to\infty}{\sim} \log (N \delta_N) \]
since  \[ \left\vert  \sum_{j=\lceil 1/\delta_N \rceil}^{N-1}  \frac{(1+\{ j\delta_N \})/\delta_N}{j(j+1)} \right\vert  \leq  \frac{2}{\delta_N} \sum_{j=\lceil 1/\delta_N \rceil}^{N-1} \frac{1}{j(j+1)} = \frac{2}{\delta_N} \left( \frac{1}{\lceil 1/\delta_N \rceil} - \frac{1}{N}\right) = \mathcal{O}(1).\]
Consequently,
\[\sum_{j=\lceil 1/\delta_N \rceil}^{N-1} \frac{A_j}{j+1} \underset{N\to\infty}{\sim} \frac{1}{6} \log (N \delta_N). \]
Furthermore we notice that $\vert \widetilde{A}_N \vert \leq \delta_N \sum_{q=1}^{\lfloor 1/\delta_N \rfloor} q\delta_N \left( 1- q\delta_N \right) = \frac{1}{6} + o (1)$, thus $A_N+B_N = \mathcal{O}(1)$. \\
We deduce 
\begin{equation}\label{eq:varcase1part2}
\sum_{j=\lceil 1/\delta_N \rceil}^N \frac{\chi_N (j)}{j} = A_N + B_N + \sum_{j=\lceil 1/\delta_N \rceil}^{N-1} \frac{A_j + B_j}{j+1} \underset{N\to\infty}{\sim} \frac{1}{6} \log (N \delta_N )  
\end{equation}
Finally, combining \eqref{eq:varcase1part1} and \eqref{eq:varcase1part2} we get \eqref{eq:asymptVarXtildeMeso} for the case $\theta=1$.

\subsubsection*{\underline{General case $\theta >0$}:}

By triangular inequality, we have 
\[\left\vert  \sum_{j=1}^N  \frac{\Psi_N (j) \chi_N (j)}{j}  - \sum_{j=1}^N  \frac{\chi_N (j)}{j}  \right\vert \leq \sum_{j=1}^N \frac{1}{j} \vert \Psi_N (j) - 1 \vert .\]
Moreover, as for any given $\theta>0$ all the terms $\Psi_N (j) - 1 $ have constant sign (the sign of $1-\theta$), then, using Lemma~\ref{lem:EqPsi}, it follows
\[\sum_{j=1}^N \frac{1}{j} \vert \Psi_N (j) - 1 \vert =\vert \theta -1 \vert \sum_{j=1}^N \frac{1}{j (\theta +j -1)} \]
which converges as $N$ goes to infinity, therefore
\[\sum_{j=1}^N  \frac{\Psi_N (j) \chi_N (j)}{j} = \sum_{j=1}^N  \frac{\chi_N (j)}{j}  + \mathcal{O}_\theta (1) .\]
We deduce from \eqref{eq:varcase1part1} and \eqref{eq:varcase1part2} that
\[\sum_{j=1}^N  \frac{\Psi_N (j) \chi_N (j)}{j} \underset{N\to \infty}{=} \frac{1}{6} \log (N\delta_N) + o (\log N\delta_N),\]
which gives \eqref{eq:asymptVarXtildeMeso} for all $\theta>0$.
\end{proof}

\begin{proposition}
Suppose that the sequence $(\delta_N)$ satisfies $\left\{\begin{array}{l} 
\delta_N \underset{N\to \infty}{\longrightarrow} 0 \\
N \delta_N \underset{N\to \infty}{\longrightarrow} + \infty .
\end{array} \right.$ Suppose in addition that $(\alpha_N)$ is a constant sequence, say $\alpha_N = \alpha$ for all $N$. Then
\begin{equation}\label{eq:asymptVarXMeso}
\mathrm{Var} (X_N^{I_N}) \underset{N\to\infty}{\sim}  \left\{\begin{array}{ll}
\frac{\theta}{6} \log(N\delta_N) & \text{if } \alpha \text{ is irrationnal}\\
\theta \left( \frac{1}{6} + \frac{1}{6q^2} \right) \log(N\delta_N) & \text{if } \alpha=\frac{p}{q} \text{ with } p,q \text{ coprime integers}, q\geq 1.
\end{array}\right.
\end{equation}
\end{proposition}
\begin{proof}
For $1\leq j \leq N$, let $\omega_j^{(N)} := \{ j \beta_N\}- \{ j\alpha_N \} = \{ j (\alpha + \delta_N) \} - \{ j\alpha \}$.
We recall that the variance at step $N$ is given by
\begin{align*}
\mathrm{Var} (X_N^{I_N} ) &= \theta \sum_{j=1}^N \frac{(\omega_j^{(N)})^2}{j} \Psi_N (j)  \\
&\qquad +\theta^2\sum_{1\leq j,k \leq N} \frac{\omega_j^{(N)} \omega_k^{(N)}}{jk} (\Psi_N (j+k) \mathds{1}_{j+k\leq N} - \Psi_N (j) \Psi_N (k)) .  
\end{align*}
Since all the $\omega_j^{(N)}$ are bounded by $1$, we have
\begin{align*}
&\left\vert \sum_{1\leq j,k \leq N} \frac{\omega_j^{(N)} \omega_k^{(N)}}{jk} (\Psi_N (j+k) \mathds{1}_{j+k\leq N} - \Psi_N (j) \Psi_N (k)) \right\vert \\
&\qquad \qquad \leq \sum_{1\leq j,k \leq N} \frac{1}{jk} \left\vert \Psi_N (j+k) \mathds{1}_{j+k\leq N} - \Psi_N (j) \Psi_N (k) \right\vert \\
	&\qquad \qquad = \mathcal{O}_\theta (1) 
\end{align*}   
by Lemma~\ref{lem:borne}, so that
\begin{equation}\label{eq:asymptVarXMesoStep1}
\mathrm{Var} (X_N^{I_N} ) = \theta \sum_{j=1}^N \frac{(\omega_j^{(N)})^2}{j} \Psi_N (j) + \mathcal{O}_\theta (1). 
\end{equation}

Assume \underline{$\alpha$ to be irrational}. The result derives from the following lemma:
\begin{lemma}
Let $t\in \mathbb{R} \setminus \mathbb{Q}$. Let $(\varepsilon_n)_{n\geq 1}$ be a sequence of positive real numbers which converges to $0$. Let $f$ be a real-valued continuous function on $[0,1]$. Then
\[\varepsilon_n \sum_{j=1}^{ \lceil \frac{1}{\varepsilon_n} \rceil -1} f(j\varepsilon_n) \mathds{1}_{j\varepsilon_n \geq 1- \{jt \} } \underset{n\to\infty}{\longrightarrow} \int_0^1 xf(x)\mathrm{d}x.   \] 
\end{lemma}
\begin{proof}
For all $n\geq 1$, let $\mu_n:= \varepsilon_n \sum_{j=1}^{\lceil \frac{1}{\varepsilon_n} \rceil -1} \delta_{(j\varepsilon_n , jt)}$ a measure on the torus $\mathbb{R}/\mathbb{Z} \times \mathbb{R}/\mathbb{Z}$. For all $(k,l)\in\mathbb{Z}^2$, the Fourier transform of $\mu_n$ in $(k,l)$ is given by
\[\widehat{\mu}_n (k,l) =\varepsilon_n \sum_{j=1}^{ \lceil \frac{1}{\varepsilon_n} \rceil -1} \left(\mathrm{e}^{2i\pi (k \varepsilon_n + lt)}\right)^j \] 
which converges to $1$ if $(k,l)=(0,0)$ and to $0$ otherwise as $n$ goes to infinity. Thus $(\mu_n)$ converges weakly to the uniform measure on $\mathbb{R}/\mathbb{Z} \times \mathbb{R}/\mathbb{Z}$. Let $f$ be a real-valued continuous function on $[0,1]$ and let $g$ be a function from the torus to $\mathbb{R}$ defined by $g(x,y) = f(x)\mathds{1}_{x\geq 1-y}$. Then $g$ is continuous everywhere excepted at most on $x=0$, $y=0$ and $x=1-y$. Consequently, the set of discontinuities of $g$ is at most one-dimensional, which is of measure zero for the Lebesgue measure of dimension $2$. Hence
\[\lim_{n\to\infty} \varepsilon_n \sum_{j=1}^{\lceil \frac{1}{\varepsilon_n} \rceil -1} g (j\varepsilon_n , jt) = \int_0^1 \int_0^1 g(x,y) \mathrm{d}x \mathrm{d}y = \int_0^1 x f(x) \mathrm{d}x. \]
\end{proof}

Noticing that $\omega_j^{(N)} = \{j\delta_N \} - \mathds{1}_{ \{ j \delta_N \} \geq 1 - \{ j \alpha \}}$ and following the same scheme than previously (treating the case $\theta =1$ and then the case $\theta >0$), we get
\[\theta \sum_{j=1}^N \frac{(\omega_j^{(N)})^2}{j} \Psi_N (j) \underset{N\to \infty } = \theta \log (N\delta_N) \delta_N \sum_{q=1}^{\lceil \frac{1}{\delta_N} \rceil} (q\delta_N -  \mathds{1}_{ q\delta_N  \geq 1 - \{ q \alpha \} } )^2 + o (\log (N\delta_N) ) \] 
with 
\begin{align*}
\delta_N \sum_{q=1}^{\lceil \frac{1}{\delta_N} \rceil} (q\delta_N -  \mathds{1}_{ q\delta_N  \geq 1 - \{ q \alpha \} } )^2 &= \delta_N \sum_{q=1}^{\lceil \frac{1}{\delta_N} \rceil} \left[ (q\delta_N)^2 +  (1 - 2 q\delta_N ) \mathds{1}_{ q\delta_N  \geq 1 - \{ q \alpha \} } \right] \\
	&\underset{N\to\infty}{\longrightarrow} \int_0^1 (x^2 + x(1-2x))\mathrm{d}x  = \frac{1}{6},
\end{align*}
which, jointly to \eqref{eq:asymptVarXMesoStep1}, gives \eqref{eq:asymptVarXMeso} for the irrational case.

Assume \underline{$\alpha$ to be rational}, say $\alpha = \frac{p}{q}$ with $p,q$ coprime numbers and $q\geq 1$. \\
Let us define the function $f:(x,y) \mapsto (x-\mathds{1}_{x\geq 1-y})^2$ on $[0,1]^2$, in order to write $(\omega_j^{(N)})^2=f(\{ j\delta_N \} , \{j\alpha \})$ for all $j$. We slightly adapt the previous proof, starting again with the case $\theta =1$. \\
We decompose into three parts 
\[\sum_{j=1}^N \frac{1}{j} f(\{ j\delta_N \} , \{j\alpha \}) = \sum_{j=1}^{\left\lceil \frac{1}{\delta_N} \right\rceil -1} \frac{1}{j} f( j\delta_N  , \{j\alpha \}) + (A_N + B_N) +  \sum_{j=\left\lceil \frac{1}{\delta_N} \right\rceil}^{N-1} \frac{1}{j+1} (A_j + B_j) \]
with for all $j\geq 1$, 
\[A_j:= \frac{1}{j} \sum\limits_{k=\left\lceil \frac{1}{\delta_N} \right\rceil}^{\lceil \lfloor j\delta_N \rfloor / \delta_N \rceil -1 }   f(\{ k\delta_N \} , \{k\alpha \})\] 
and 
\[B_j :=\frac{1}{j} \sum\limits_{k=\lceil \lfloor j\delta_N \rfloor / \delta_N \rceil}^j   f(\{ k\delta_N \} , \{k\alpha \}).\]
For the first part, noticing that for all $x\in \left[ 0,\frac{1}{2} \right]$ and $j \geq 1$ we have $f(x, \{ j\alpha \} ) \leq f(x , \frac{q-1}{q})$, then
\begin{align*}
0\leq \sum_{j=1}^{\left\lceil \frac{1}{\delta_N} \right\rceil -1} \frac{1}{j} f( j\delta_N  , \{j\alpha \}) &\leq  \sum_{j=1}^{\left\lceil \frac{1}{2\delta_N} \right\rceil -1} \frac{1}{j} f\left( j\delta_N  , \frac{q-1}{q}\right) +  \sum_{j=\left\lceil \frac{1}{2\delta_N} \right\rceil}^{\left\lceil \frac{1}{\delta_N} \right\rceil -1} \frac{1}{j} \\
&\underset{N\to\infty}{\longrightarrow} \int_0^{1/2} \frac{(x-\mathds{1}_{x\geq \frac{1}{q}})^2}{x} \mathrm{d}x + \int_{1/2}^1 \frac{1}{x}\mathrm{d}x   <+\infty.  
\end{align*} 
For the third part, since $f$ is bounded by $1$ on $[0,1]^2$, it is easy to check that \\ $\sum\limits_{j=\lceil \frac{1}{\delta_N} \rceil}^{N-1} \frac{1}{j+1}  B_j = \mathcal{O} (1)$. Moreover, $A_j$ for $j\geq 1$ can be formulated as
\begin{align*}
A_j &= \frac{1}{j} \sum_{l=1}^{\left\lfloor j\delta_N \right\rfloor -1} \left(\sum_{k=\left\lceil \frac{l}{q\delta_N} \right\rceil}^{\left\lfloor \frac{1}{q} \left\lceil \frac{l+1}{\delta_N} \right\rceil \right\rfloor -1}  \sum_{m=0}^{q-1} f_m ( \{ (kq+m)\delta_N  \} )\right) 
\\
	&\qquad + \sum_{r = \left\lceil \frac{l}{\delta_N} \right\rceil}^{q \left\lceil \frac{l}{q\delta_N} \right\rceil -1} f_r (\{r\delta_N \} ) + \sum_{r = q \left\lfloor \frac{1}{q} \left\lceil \frac{l+1}{\delta_N} \right\rceil \right\rfloor }^{\left\lceil \frac{l+1}{\delta_N} \right\rceil -1} f_r (\{r\delta_N \}) 
\end{align*}
where $f_n (x):= f(x,\{n\alpha \})$ for all $n$. This new expression of $A_j$ is manageable. Indeed, 
\begin{align*}
\frac{1}{j} \sum_{l=1}^{\left\lfloor j\delta_N \right\rfloor -1} \left( \sum_{r = \left\lceil \frac{l}{\delta_N} \right\rceil}^{q \left\lceil \frac{l}{q\delta_N} \right\rceil -1} f_r (\{r\delta_N \} ) + \sum_{r = q \left\lfloor \frac{1}{q} \left\lceil \frac{l+1}{\delta_N} \right\rceil \right\rfloor }^{\left\lceil \frac{l+1}{\delta_N} \right\rceil -1} f_r (\{r\delta_N \})\right) &\leq \frac{1}{j}\left(\left\lfloor j\delta_N \right\rfloor -1 \right)(q+q) \\
	&\leq 2q \delta_N, 
\end{align*}
and for all $l, m$, 
\[\sum_{k=\left\lceil \frac{l}{q\delta_N} \right\rceil}^{\left\lfloor \frac{1}{q} \left\lceil \frac{l+1}{\delta_N} \right\rceil \right\rfloor -1}  f_m ( \{ (kq+m)\delta_N  \}) = \sum_{k=0}^{\left\lfloor \frac{1}{q} \left\lceil \frac{l+1}{\delta_N} \right\rceil \right\rfloor -1 - \left\lceil \frac{l}{q\delta_N} \right\rceil}  f_m \left(  \left(kq+m +q\left\{ \frac{-l}{q\delta_N}  \right\} \right)\delta_N  \right) .\]
Let $\widetilde{A}_j := \frac{1}{j} (\lfloor j\delta_N \rfloor -1) \sum\limits_{m=0}^{q-1} \sum\limits_{k=0}^{\left\lfloor \frac{1}{q\delta_N}  \right\rfloor -1 } f_m ((kq+m)\delta_N)$. Then
\begin{align*}
\vert A_j - \widetilde{A}_j \vert &\leq \frac{1}{j} \sum_{l=1}^{\left\lfloor j\delta_N \right\rfloor -1}   \sum_{m=0}^{q-1} \sum_{k=0}^{\left\lfloor \frac{1}{q\delta_N}  \right\rfloor -1 } \left\vert  f_m \left(  \left(kq+m +q\left\{ \frac{-l}{q\delta_N}  \right\} \right)\delta_N  \right) -  f_m \left( (kq+m)\delta_N  \right) \right\vert \\
	&\qquad + \mathcal{O} (\delta_N).
\end{align*}
For all $0\leq m\leq q-1$, the function $f_m$ is piecewise polynomial on $[0,1]$ (it has one finite discontinuity at point $t=1-\{m\alpha\}$). Thus
\begin{align*}
&\left\vert  f_m \left(  \left(kq+m +q\left\{ \frac{-l}{q\delta_N}  \right\} \right)\delta_N  \right) -  f_m \left( (kq+m)\delta_N  \right) \right\vert \\ 
&\qquad \qquad  \leq q\left\{ \frac{-l}{q\delta_N}  \right\} \delta_N \esssup_{[0,1]} \vert f_m^\prime \vert \\
	&\qquad \qquad \qquad + \mathds{1}_{t \leq \left( kq+m + q\left\{ \frac{-l}{q\delta_N}  \right\} \right)\delta_N < t+ q\left\{ \frac{-l}{q\delta_N} \right\} \delta_N     }  \\ 
	&\qquad \qquad \leq 2 q\delta_N + \mathds{1}_{t \leq \left( kq+m + q\left\{ \frac{-l}{q\delta_N}  \right\} \right)\delta_N < t+ q\left\{ \frac{-l}{q\delta_N} \right\} \delta_N   }. 
\end{align*}
At fixed $l,m$, it is easy to check that there is at most one $k$ such that the indicator function equals $1$. Hence
\[ \vert A_j - \widetilde{A}_j \vert \leq \frac{1}{j} \times j\delta_N \times q \times \left[ \frac{1}{q\delta_N} \times (2q \delta_N) +1   \right] + \mathcal{O} (\delta_N) = \mathcal{O} (\delta_N)\]
and consequently
\[\left \vert \sum_{j=\lceil 1/\delta_N \rceil}^{N-1} \frac{A_j}{j+1} -  \sum_{j=\lceil 1/\delta_N \rceil}^{N-1} \frac{\widetilde{A}_j}{j+1}  \right\vert = o (\log (N\delta_N)). \]
Moreover, 
\begin{align*}
\sum_{j=\lceil 1/\delta_N \rceil}^{N-1} \frac{\widetilde{A}_j}{j+1} &= \delta_N \sum_{m=0}^{q-1}  \sum_{k=0}^{\left\lfloor \frac{1}{q\delta_N}  \right\rfloor -1 } f_m ((kq+m)\delta_N) \left( \sum_{j=\lceil 1/\delta_N \rceil}^{N-1} \frac{1}{j+1}  + \mathcal{O}(1) \right) \\
	&= \frac{1}{q} \sum_{m=0}^{q-1} \left( \int_0^1 f_m (x) \mathrm{d}x + o(1)  \right) (\log(N\delta_N) + \mathcal{O}(1)).
\end{align*}
Since $p$ and $q$ are coprime numbers, the numbers $\left\{ m \alpha \right\}$ cycle through some rearrangement of the numbers $0 , \frac{1}{q} ,\cdots , \frac{q-1}{q}$, hence
\begin{align*}
\frac{1}{q} \sum_{m=0}^{q-1} \int_0^1 f_m (x) \mathrm{d}x &= \frac{1}{q} \sum_{m=0}^{q-1} \int_0^1 \left( x-\mathds{1}_{x\geq 1- \frac{m}{q}} \right)^2 \mathrm{d}x \\
 	&= \frac{1}{3} - \frac{1}{q} \sum_{m=1}^{q-1} \frac{m}{q} \left( 1- \frac{m}{q} \right) \\
 	&= \frac{1}{6} + \frac{1}{6q^2}
\end{align*}
(indeed when $q=1$ the last equality is satisfied too). \\
Furthermore, note that $\vert \widetilde{A}_N \vert \leq \delta_N \sum\limits_{m=0}^{q-1} \sum\limits_{k=0}^{\left\lfloor \frac{1}{q\delta_N}  \right\rfloor -1 } 1 \leq 1$ and then $\vert A_N + B_N \vert = \mathcal{O} (1)$.  

Finally, putting it all together, we have
\begin{equation}\label{eq:asymptVarXMesoTheta1}
\sum_{j=1}^N \frac{(\omega_j^{(N)})^2}{j} = \left(\frac{1}{6} + \frac{1}{6q^2} \right) \log(N\delta_N) + o (\log (N \delta_N))
\end{equation}
and using the same argument as in the previous proof we extend \eqref{eq:asymptVarXMesoTheta1} to all $\theta>0$, which provides
\[\sum_{j=1}^N \frac{(\omega_j^{(N)})^2}{j} \Psi_N (j) = \left(\frac{1}{6} + \frac{1}{6q^2} \right) \log(N\delta_N) + o (\log (N \delta_N)) \]
From \eqref{eq:asymptVarXMesoStep1} we deduce \eqref{eq:asymptVarXMeso} for the rational case.
\end{proof}

\subsection{Limiting normality for a single mesoscopic arc}

\everymath{\displaystyle}
\begin{theorem}\label{th:th:TCLMesowreath}
Suppose that the sequence $(\delta_N)$ satisfies $\left\{\begin{array}{l} 
\delta_N \underset{N\to \infty}{\longrightarrow} 0 \\
N \delta_N \underset{N\to \infty}{\longrightarrow} + \infty .
\end{array} \right.$ \\
Let $\widetilde{Y}_N^{I_N} := \frac{\widetilde{X}_N^{I_N} - \mathbb{E} (\widetilde{X}_N^{I_N}) }{\left( \frac{\theta}{6} \log (N \delta_N) \right)^{1/2}}$. Then
\[\widetilde{Y}_N^{I_N}  \overset{\text{d}}{\longrightarrow} Z  \] where $Z$ is a centred Gaussian random variable of variance $1$.
\end{theorem}
\everymath{}
\begin{proof}
This proof is very similar to the proof of Theorem~\ref{th:TCLWreath}, in the particular case of one interval. \\
We introduce again the Poisson variables $(W_j)$ from Feller Coupling. We denote for all $j,p$, $b_{j,p}^{(N)} := \mathds{1}_{\phi_{j,p} > \{j\alpha_N \}} -  \mathds{1}_{\phi_{j,p} > \{j\beta_N \}}+ \{j\alpha_N \} - \{j\beta_N \}$, and $T_N := \sum_{j=1}^N V_{N,j}$ with for all $j$, 
\[V_{N,j} := \frac{1}{\left( \frac{\theta}{6} \log (N \delta_N) \right)^{1/2}} \sum_{p=1}^{W_j} b_{j,p}^{(N)}.\]  
Let $\varepsilon >0$. Using Markov's inequality and Lemma~\ref{lem:majPoisson},
\begin{align*}
\mathbb{P} \left( \vert \widetilde{Y}_N^{I_N} - T_N  \vert > \varepsilon \right) &\leq \frac{1}{\varepsilon \left( \frac{\theta}{6} \log (N \delta_N)  \right)^{1/2}} \sum_{j=1}^N \mathbb{E} \sum_{p=1}^{\vert a_{N,j} -W_j \vert} \vert b_{j,p}^{(N)} \vert \\
	&= \frac{1}{\varepsilon \left( \frac{\theta}{6} \log (N \delta_N)  \right)^{1/2}} \times \mathcal{O}_\theta (1).
\end{align*}
Hence $ \widetilde{Y}_N^{I_N} - T_N \overset{\mathbb{P}}{\longrightarrow} 0$.
Moreover, for all $j$, 
\[\mathrm{Var} (V_{N,j} ) = \frac{1}{\frac{1}{6} \log (N \delta_N) }  \frac{1}{j} \{ j\delta_N \} ( 1- \{ j\delta_N \} ).\]
In a similar way as in the proof of Theorem~\ref{th:TCLWreath}, we have 
\[ \sum_{j=1}^N \mathrm{Var} (V_{N,j}) = \frac{6}{ \log (N \delta_N) } \sum_{j=1}^N  \frac{ \{ j\delta_N \} ( 1- \{ j\delta_N \} ) }{j} \underset{N\to \infty}{\longrightarrow} 1\]
and
\[\sum_{j=1}^N \mathbb{E} \left( (V_{N,j})^2 \mathds{1}_{\vert V_{N,j} \vert > \varepsilon}  \right) \leq \frac{1}{\frac{\theta}{6} \log (N \delta_N) }  \sum_{j=1}^N \mathbb{E} \left( W_j^2 \mathds{1}_{W_j > \varepsilon \left( \frac{\theta}{6} \log (N \delta_N)   \right)^{1/2}}  \right) \underset{N\to \infty}{\longrightarrow} 0.\]
Then $(T_N)$ converges in distribution to $\mathcal{N}(0,1)$. Slutsky's theorem ends the proof.
\end{proof}

\everymath{\displaystyle}
\begin{theorem}\label{th:TCLMeso}
Suppose that $I_N= \left( \mathrm{e}^{2i\pi \alpha} , \mathrm{e}^{2i\pi (\alpha + \delta_N)} \right]$ with $\left\{\begin{array}{l} 
\delta_N \underset{N\to \infty}{\longrightarrow} 0 \\
N \delta_N \underset{N\to \infty}{\longrightarrow} + \infty .
\end{array} \right.$ \\
Let $Y_N^{I_N} := \frac{X_N^{I_N} - \mathbb{E} (X_N^{I_N}) }{\left( \theta c_2 (\alpha)  \log (N \delta_N) \right)^{1/2}}$ where $c_2 (\alpha)$ is a constant defined by
\[c_2 (\alpha):= \left\{ \begin{array}{cl}
	\frac{1}{6} &\text{ if $\alpha$ is irrational} \\
	\frac{1}{6} + \frac{1}{6q^2} &\text{ if $\alpha = \frac{p}{q}$, $p,q$ coprime numbers, $q\geq 1$.}
\end{array}  
\right. \]
Then
\[Y_N^{I_N}  \overset{\text{d}}{\longrightarrow} Z  \] where $Z$ is a centred Gaussian random variable of variance $1$.
\end{theorem}
\everymath{}
\begin{proof}
This proof is very similar to the proof of Theorem~\ref{th:TCLfix}, in the particular case of one interval.  
We introduce again the Poisson variables $(W_j)$ from Feller Coupling. Let $T_N := \sum_{j=1}^N V_{N,j}$ with for all $j$, 
\[V_{N,j} := \frac{ \frac{\theta}{j} - W_j }{\left( \theta c_2 (\alpha)  \log (N \delta_N) \right)^{1/2}}  \omega_j^{(N)}.\]  
Let $\varepsilon >0$. Using Markov's inequality, Lemma~\ref{lem:EqPsi} and Lemma~\ref{lem:majPoisson},
\begin{align*}
\mathbb{P} \left( \vert Y_N^{I_N} - T_N \vert > \varepsilon \right) &\leq \frac{1}{\varepsilon \left( \theta c_2 (\alpha)  \log (N \delta_N) \right)^{1/2}} \left[ \sum_{j=1}^N \frac{\theta}{j} \vert \Psi_N (j) - 1 \vert   + \sum_{j=1}^N \mathbb{E} \vert W_j - a_{N,j} \vert \right] \\
	&= \frac{1}{\varepsilon \left( \theta c_2 (\alpha)  \log (N \delta_N) \right)^{1/2}} \times \mathcal{O}_\theta (1),
\end{align*}
hence $ Y_N^{I_N} - T_N \overset{\mathbb{P}}{\longrightarrow} 0$. 
In a similar way as in the proof of Theorem~\ref{th:TCLfix}, we have 
\[\sum_{j=1}^N \mathrm{Var} (V_{N,j}) = \frac{1}{c_2 (\alpha) \log (N \delta_N) } \sum_{j=1}^N  \frac{ (\omega_j^{(N)})^2}{j} \underset{N\to \infty}{\longrightarrow} 1\]
and
\begin{align*}
\sum_{j=1}^N \mathbb{E} \left( (V_{N,j})^2 \mathds{1}_{\vert V_{N,j} \vert > \varepsilon}  \right) &\leq \frac{1}{\theta c_2 (\alpha)  \log (N \delta_N) }  \sum_{j=1}^N \mathbb{E} \left( \left(W_j - \frac{\theta}{j} \right)^2 \mathds{1}_{\left\vert W_j - \frac{\theta}{j} \right\vert > \varepsilon \left( \theta c_2 (\alpha)  \log (N \delta_N) \right)^{1/2}}  \right) \\
&\underset{N\to \infty}{\longrightarrow} 0.
\end{align*}
Then $(T_N)$ converges in distribution to $\mathcal{N}(0,1)$. Slutsky's theorem ends the proof.
\end{proof}

\begin{remark}
For a finite number of mesoscopic arcs, say $m$ arcs, with shrinking speeds $\delta_N^{(k)}$, $1\leq k \leq m$, it is reasonable to expect that some asymptotic results still occur. Indeed, the only point to overcome in the proof is the existence of non-diagonal terms in the covariance matrices $D$ and $\widetilde{D}$, whose good candidates would be for $k\neq l$: 
\begin{equation}
D_{k,l}=\frac{1}{(c_2 (\alpha_k) c_2 (\alpha_l))^{1/2}} \lim_{N\to\infty} \frac{1}{\left( \log (N \delta_N^{(k)}) \log ( N \delta_N^{(l)} ) \right)^{1/2}} \sum_{j=1}^N \frac{\omega_{j,k}^{(N)}\omega_{j,l}^{(N)}}{j}
\end{equation}
and
\begin{equation}
\widetilde{D}_{k,l}=6 \lim_{N\to\infty} \frac{1}{\left( \log (N \delta_N^{(k)}) \log ( N \delta_N^{(l)} ) \right)^{1/2}} \sum_{j=1}^N \frac{H_{j,k,l}^{(N)}}{j},
\end{equation}
where 
\[ \left\{\begin{array}{l} 
\omega_{j,k}^{(N)} := \{j(\alpha_k + \delta_N^{(k)}) \} - \{ j\alpha_k \} \\
H_{j,k,l} = \frac{1}{2} ( h_j (\alpha_N^{(k)} - \beta_N^{(l)} )+ h_j (\beta_N^{(k)} - \alpha_N^{(l)} )  - h_j (\alpha_N^{(k)} - \alpha_N^{(l)} ) - h_j (\beta_N^{(k)} - \beta_N^{(l)} )),
\end{array}\right.
\]
$h_j (x) := \{jx \}(1 - \{jx\})$.
It is not clear these limits exist, since the formulas suggest a deep dependence on the way the $m$ arcs overlap when $N$ becomes large. 
\end{remark}

\section{Spacing between eigenvalues}
\label{sec:5}

For all $N\geq 1$, denote by $\mathcal{D}_N$ and $d_N$ (resp. $\widetilde{\mathcal{D}}_N$ and $\widetilde{d}_N$) the largest and the smallest spacings between two consecutive distinct eigenangles of a random element from $\mathfrak{S}_N$ (resp. $S^1 \wr \mathfrak{S}_N$), where the permutations are picked under Ewens measure of parameter $\theta>0$.

\subsection{Largest spacing between two consecutive distinct eigenvalues}

\begin{proposition}
The sequences of random variables $(n \mathcal{D}_n)_{n\geq 1}$, $(n \widetilde{\mathcal{D}}_n)_{n\geq 1}$, $(\frac{1}{n \mathcal{D}_n})_{n\geq 1}$ and $(\frac{1}{n \widetilde{\mathcal{D}}_n})_{n\geq 1}$ are tight.
\end{proposition}

\begin{remark}
Informally, this proposition involves that $\mathcal{D}_n$ and $\widetilde{\mathcal{D}}_n$ have an order of magnitude of $\frac{1}{n}$.
\end{remark}

\begin{proof}
Let $n\geq 1$. It is easy to check that
\[\frac{2\pi}{Z_n} \leq \mathcal{D}_n	\leq \frac{2\pi}{L_{n,1}}\]
where $L_{n,1}$ denotes the largest cycle length of the corresponding permutation and $Z_n$ the number of distinct eigenvalues, using the pigeonhole principle for the first inequality. In particular, since $Z_n \leq n$,
\begin{equation}\label{eq:nDn}
2 \pi \leq n \mathcal{D}_n \leq \frac{2\pi}{\frac{1}{n} L_{n,1}}.
\end{equation}
Obviously, the left-hand side of \eqref{eq:nDn} provides the tightness of $(\frac{1}{n \mathcal{D}_n})_{n\geq 1}$.
Moreover, it is well-known (see \cite{arratia2003logarithmic}) that $\left(\frac{1}{n} L_{n,1}\right)$ converges in distribution to the first coordinate of a Poisson-Dirichlet random vector of parameter $\theta$, which is almost surely finite and positive. Using the continuous mapping theorem, it follows that $\left( \frac{2\pi}{\frac{1}{n} L_{n,1}} \right)$ converges in distribution, and thus this sequence is tight. We deduce
\[ \forall \varepsilon >0 , \quad \exists \eta_\varepsilon >0,  \quad \forall n \geq 1, \quad \mathbb{P} (n \mathcal{D}_n \leq \eta_\varepsilon) \geq  \mathbb{P} \left( \frac{2\pi}{\frac{1}{n} L_{n,1}} \leq \eta_\varepsilon \right) \geq 1-\varepsilon. \] 
Now, for a random element of the wreath product $S^1 \wr \mathfrak{S}_n$ related to the same permutation, the previous inequality holds, \emph{i.e} 
\begin{equation}\label{eq:nDtilden}
\frac{2\pi}{n} \leq \widetilde{\mathcal{D}}_n	\leq \frac{2\pi}{L_{n,1}}.
\end{equation}
Note that in this case the number of distinct eigenvalues is almost surely equal to $n$.
The same reasoning as above applied to $\widetilde{\mathcal{D}}_n$ gives the claim.
\end{proof}

\subsection{Smallest spacing between two consecutive distinct eigenvalues}

\begin{proposition}
The sequences of random variables $(n^2 d_n)_{n\geq 1}$, $(n^2 \widetilde{d}_n)_{n\geq 1}$, $(\frac{1}{n^2 d_n})_{n\geq 1}$ and $(\frac{1}{n^2 \widetilde{d}_n})_{n\geq 1}$ are tight.
\end{proposition}

\begin{remark}
Informally, this proposition involves that $d_n$ and $\widetilde{d}_n$ have an order of magnitude of $\frac{1}{n^2}$.
\end{remark}

\begin{proof}
Let $(a_{n,1} , \cdots , a_{n,n})$ be the cycle structure and $(A_{n,1} , \cdots , A_{n,n})$ be the \\ \emph{age-ordered list of cycle lengths} (\emph{i.e} the vector of cycle lengths in order of appearance, that is to say the increasing order following the lowest element of each cycle) of the random $n$-permutation. \\
The smallest spacing $d_n$ can be formulated as
\begin{equation}
d_n := \frac{2\pi}{\sup \{ \mathrm{lcm} (k,l) \  : \ 1\leq k , l \leq n , \quad a_{n,k} , a_{n,l} \geq 1  \}  } 
\end{equation}
On the one hand, trivially $d_n \geq \frac{2\pi}{n^2}$ for all $n$, then $(\frac{1}{n^2 d_n})_{n\geq 1}$ is tight. \\
On the other hand, 
\begin{equation}\label{eq:ineqdn}
d_n  \leq \frac{2\pi}{\mathrm{lcm} (A_{n,1}, A_{n,2})} = \frac{2\pi \cdot \gcd (A_{n,1}, A_{n,2})}{A_{n,1} A_{n,2}}.
\end{equation}
We are going to show that $(\gcd (A_{n,1}, A_{n,2}))_{n\geq 1}$ is tight. \\
Let $A$ be a positive integer. 
\begin{align*}
\mathbb{P} (\gcd (A_{n,1}, A_{n,2}) \geq A) &= \sum_{j=A}^n \mathbb{P} (\gcd (A_{n,1}, A_{n,2}) = j) \\
	&\leq \sum_{j=A}^n \mathbb{P} (j \text{ is a common divisor of } A_{n,1} \text{ and } A_{n,2}) \\
	&= \sum_{j=A}^n \sum_{\substack{1\leq k,l \\ k+l \leq \lfloor \frac{n}{j} \rfloor }} \mathbb{P} ((A_{n,1} , A_{n,2})=(jk,jl)).
\end{align*}
From a basic result in \cite{arratia2003logarithmic} we have, for all $a_1,a_2 \geq 1$, 
\[\mathbb{P} (A_{n,1} = a_1 , A_{n,2} = a_2 ) = \frac{\theta^2 \Psi_n (a_1+a_2)}{n (n-a_1)}.\]
Thus, for all $j\geq 1$,
\[ \sum_{\substack{1\leq k,l \\ k+l \leq \lfloor \frac{n}{j} \rfloor }} \mathbb{P} ((A_{n,1} , A_{n,2})=(jk,jl)) = \frac{\theta^2}{n} \sum_{k=1}^{\lfloor \frac{n}{j} \rfloor -1} \frac{1}{n-jk} \sum_{l=1}^{\lfloor \frac{n}{j} \rfloor  - k} \Psi_n (j(k+l)). \] 
\underline{If $\theta \geq 1$}, then $\Psi_n (m) \leq 1$ for all $1\leq m\leq n$, and then
\[\frac{\theta^2}{n} \sum_{k=1}^{\lfloor \frac{n}{j} \rfloor -1} \frac{1}{n-jk} \sum_{l=1}^{\lfloor \frac{n}{j} \rfloor  - k} \Psi_n (j(k+l)) \leq \frac{\theta^2}{n} \sum_{k=1}^{\lfloor \frac{n}{j} \rfloor -1} \frac{1}{n-jk} \left(\left\lfloor \frac{n}{j} \right\rfloor  - k \right) \leq \frac{\theta^2}{j^2}.\]
\underline{If $\theta <1$}, then $\Psi_n (m) \leq \Psi_n (m+1) $ for all $1\leq m\leq n-1$, and then for all $j\geq 1$ we have
\begin{align*}
\sum_{l=1}^{\lfloor \frac{n}{j} \rfloor -k-1} \Psi_n (j(k+l)) &\leq \frac{1}{j} \sum_{l=1}^{\lfloor \frac{n}{j} \rfloor  - k-1} \sum_{p=jl}^{j(l+1)-1}  \Psi_{n} (jk + p) \\
	&\leq \Psi_n (jk) \frac{1}{j} \sum_{p=1}^{n-jk} \Psi_{n-jk} (p) \\
	&= \Psi_n (jk) \frac{n-jk}{j\theta},
\end{align*}
using Lemma~\ref{lem:EqPsi} for the latest equality, and similarly
\begin{align*}
\sum_{k=1}^{\lfloor \frac{n}{j} \rfloor -1} \Psi_n (jk) &\leq \frac{1}{j} \sum_{k=1}^{\lfloor \frac{n}{j} \rfloor  -1} \sum_{p=jk}^{j(k+1)-1}  \Psi_{n} (p) \leq \frac{n}{j\theta}.
\end{align*}
Thus, 
\begin{align*}
\frac{\theta^2}{n} \sum_{k=1}^{\lfloor \frac{n}{j} \rfloor -1} \frac{1}{n-jk} \sum_{l=1}^{\lfloor \frac{n}{j} \rfloor  - k} \Psi_n (j(k+l)) &\leq \frac{\theta^2}{n} \left[ \left( \frac{1}{j\theta} \times \frac{n}{j\theta } \right) + \sum_{k=1}^{\lfloor \frac{n}{j} \rfloor -1}  \frac{1}{n-jk} \Psi_n \left(j \left\lfloor \frac{n}{j} \right\rfloor \right) \right] \\
&\leq \frac{1}{j^2} + \frac{\theta^2}{n} \Psi_n (n) \sum_{k=1}^{\lfloor \frac{n}{j} \rfloor -1} \frac{1}{n-jk}   
\end{align*}
with, using Holder inequality,
\begin{align*}
\sum_{k=1}^{\lfloor \frac{n}{j} \rfloor -1} \frac{1}{n-jk} \leq \frac{1}{j} \sum_{k=0}^{\lfloor \frac{n}{j} \rfloor -1} \frac{1}{\left\lfloor \frac{n}{j} \right\rfloor - k} =  \frac{1}{j} \sum_{k=1}^{\lfloor \frac{n}{j} \rfloor } \frac{1}{k} 
	&\leq \frac{1}{j} \left( \sum_{k=1}^{\lfloor \frac{n}{j} \rfloor } \frac{1}{k^{1/(1-\theta)}}   \right)^{1-\theta}  \left( \left\lfloor \frac{n}{j} \right\rfloor  \right)^\theta  \\
	&= \mathcal{O}_\theta \left( \frac{n^\theta}{j^{1+\theta}} \right) .
\end{align*}
Consequently, we deduce that for all $\theta >0$,
\begin{align*}
\mathbb{P} (\gcd (A_{n,1}, A_{n,2}) \geq A) &\leq (\theta^2 +1) \sum_{j=A}^n \frac{1}{j^2} + \mathcal{O}_\theta  \left( \sum_{j=A}^n \frac{\Psi_n (n)}{n^{1-\theta} j^{1+\theta}} \right) \\
	&\leq (\theta^2 +1) \sum_{j=A}^{+\infty} \frac{1}{j^2} + \mathcal{O}_\theta \left( \sum_{j=A}^{+\infty} \frac{1}{j^{1+\theta}} \right) \\
	&\leq \varepsilon 
\end{align*}
for all fixed real numbers $\varepsilon >0$ and $A$ large enough depending on $\theta$ and $\varepsilon$. \\
Finally, it is well-known (see again \cite{arratia2003logarithmic}) that as $n\to \infty$
\begin{equation}
\frac{1}{n} (A_{n,1} , A_{n,2} , \cdots , A_{n,n}) \overset{d}{\longrightarrow} (G_1 , G_2 , \cdots )
\end{equation}
where $(G_1, G_2 ,\cdots )$ has GEM$(\theta)$ distribution. Note that $G_1$ and $G_2$ are almost surely finite and positive. Since from \eqref{eq:ineqdn}
\[2\pi \leq n^2 d_n \leq \frac{2\pi \gcd (A_{n,1}, A_{n,2})}{\frac{1}{n} A_{n,1} \cdot \frac{1}{n} A_{n,2}},\]
then the continuous mapping theorem gives the tightness of $(n^2d_n)_{n\geq 1}$.

Now, for a random element of the wreath product $S^1 \wr \mathfrak{S}_n$ related to the same permutation, the spacing is lower than the one for the permutation. Indeed, for all pair of cycle lengths $(p,q)$, 
\begin{itemize}
\item if $p=q$, then applying whatever rotations on the corresponding sets of eigenvalues (the $p$-th roots of unity) decreases the distance between two distinct eigenvalues.
\item if $p\neq q$, we can suppose without loss of generality that $p$ and $q$ are coprime numbers (since the set of all eigenangles is periodic of period $2\pi/\gcd (p,q)$). Let $s\in (0,1)$. It is easy to check that the smallest spacing between two arguments of points which are respectively taken from $\left\{ \mathrm{e}^{ \frac{2i\pi k }{p}} , 0\leq k\leq p-1   \right\}$ and $\left\{ \mathrm{e}^{ 2i\pi (\frac{l}{q} + s)} , 0\leq l\leq q-1   \right\}$ divided by $2\pi$ is equal to
\begin{align*}
\min_{\substack{k \in \mathbb{Z} \\ l\in \mathbb{Z}}} \left\vert \frac{l}{q} + s - \frac{k}{p}  \right\vert = \frac{1}{pq}  \min_{\substack{k \in \mathbb{Z} \\ l\in \mathbb{Z}}} \left\vert lp -kq + spq  \right\vert \\
\end{align*}
and by Bézout we know that $\{lp - kq : k,l\in \mathbb{Z} \} = \mathbb{Z}$, thus
\begin{align*}
\min_{\substack{k \in \mathbb{Z} \\ l\in \mathbb{Z}}} \left\vert lp -kq + spq  \right\vert =   \min_{n\in \mathbb{Z}}  \vert n +spq \vert = \min (\{ spq \} , 1 - \{spq\}) \leq 1.
\end{align*}
\end{itemize}
Consequently we have $\widetilde{d}_n \leq d_n$ and we deduce the tightness of $(n^2 \widetilde{d}_n)_{n\geq 1}$.

It remains to show that $(\frac{1}{n^2 \widetilde{d}_n})_{n\geq 1}$ is tight. For this purpose, denote $E_n$ the ensemble of couples of cycle lengths of the considered randomly chosen $n$-permutation. First observe that for all uniform random variables $U$ on $[0,1]$ and all non-zero integer $n$, $\{nU\}$ is uniform on $[0,1]$ and then $\min (\{nU\} , 1- \{nU \} )$ is uniform on $[0,1/2]$. Hence, using what we did above,  
\begin{equation}\label{eq:ineqdtilden}
\widetilde{d}_n \geq 2\pi \min_{(l_1, l_2) \in E_n} \frac{1}{l_1 l_2} V_{l_1,l_2} 
\end{equation}
where $V_{j,k}$ are uniform random variables on $[0, 1/2]$ (which are not independent when the indices overlap). \\
Now, conditionally to $E_n=E$ where $E$ is a possible ensemble of couples of lengths for the picked $n$-permutation, we have for all positive real numbers $t< \frac{1}{2}$,
\begin{align*}
\mathbb{P} \left( \min_{(l_1, l_2) \in E_n} \frac{1}{l_1 l_2} V_{l_1,l_2} \leq \frac{t}{n^2} \ \vert \ E_n = E \right) &\leq \sum_{(l_1, l_2) \in E} \mathbb{P} \left( V_{l_1,l_2} \leq \frac{t l_1 l_2}{n^2}  \right) \\
&\leq \sum_{(l_1, l_2) \in E} 2 \frac{t l_1 l_2}{n^2} = 2t
\end{align*}
using the union bound for the first inequality. Thus, 
\[\mathbb{P} \left( \min_{(l_1, l_2) \in E_n} \frac{1}{l_1 l_2} V_{l_1,l_2} \leq \frac{t}{n^2} \right) \leq 2t \]
and, from \eqref{eq:ineqdtilden}, 
\[\mathbb{P} \left( \frac{1}{n^2 \widetilde{d}_n} > A \right) \leq \frac{1}{\pi A}\]
which gives the required tightness.
\end{proof}

\section*{Appendix A}

We give a proof of Lemma~\ref{lem:majPoisson} using Cesàro means. Our proof provide a slightly better upper-bound than the one given in \cite{arratia2003logarithmic}. We begin with the following lemma:

\begin{lemma}\label{lem:Psi}
For all $1\leq j\leq n$, 
\begin{equation}
\sum_{p=j}^{n-1} \frac{A_{p-j}^{\theta -1}}{p A_p^\theta} = \Psi_n (j) \left( \frac{1}{j} - \frac{1}{n} \right).
\end{equation}
\end{lemma}
\begin{proof}
With the same notation as in the proof of Proposition~\ref{prop:log}, we observe that
\begin{align*}
J_n(Ms)_n = \sum_{p=1}^{n-1} J_n M_{n,p} s_p + s_n &=  \sum_{p=1}^{n-1} \frac{\theta}{\theta +p} \frac{1}{p} \sum_{j=1}^{p} \Psi_p (j) w_j + \frac{1}{n} \sum_{j=1}^n \Psi_n (j) w_j \\
	&= \sum_{j=1}^{n-1} w_j \left[ \sum_{p=j}^{n-1} \frac{\theta}{\theta +p} \frac{\Psi_p (j)}{p} + \frac{\Psi_n (j)}{n} \right] +  \frac{\Psi_n (n)}{n} w_n
\end{align*}
and
\[J_n t_n = \sum_{j=1}^{n-1} w_j \frac{\Psi_n (j)}{j} + \frac{\Psi_n (n)}{n} w_n.\]
Therefore, by identification, for all $j\in [\![ 1 , n-1 ]\!]$, 
\[\frac{\Psi_n (j)}{j} = \sum_{p=j}^{n-1} \frac{\theta}{\theta +p} \frac{\Psi_p (j)}{p} + \frac{\Psi_n (j)}{n}.  \]
Finally, using the definition of Cesàro numbers it is clear that $\frac{\theta}{\theta +p} \frac{\Psi_p (j)}{p} = \frac{A_{p-j}^{\theta -1}}{p A_p^\theta}$ for all $1\leq j \leq p$.
\end{proof}

We are ready to prove Lemma~\ref{lem:majPoisson}.

We introduce independent Bernoulli variables $\xi_r$, $r\geq 1$, defined as
\[\mathbb{P} (\xi_r = 1)= \frac{\theta}{\theta + r -1} \quad , \quad \mathbb{P} (\xi_r =0) = \frac{r-1}{\theta+r-1}.\]
The Feller Coupling characterizes the variables $a_{n,j}$ and $W_j$ on the same probability space in function of the $\xi_r$, by the following equalities: For all $1 \leq j \leq n$,   
\begin{align*}
a_{n,j} &= \# \{j-\text{spacings between two consecutive $1$ in the word } (1 \ \xi_2 \ \cdots \ \xi_n \ 1) \} \\
 	&= \sum_{k=1}^{n-j} \xi_k (1-\xi_{k+1}) \cdots (1-\xi_{k+j-1})\xi_{k+j} + \xi_{n-j+1} (1- \xi_{n-j+1} ) \cdots (1- \xi_n ),
\end{align*}
and for all $j \in \mathbb{N}^*$, 
\[W_j = \sum_{k=1}^{+\infty} \xi_k (1-\xi_{k+1}) \cdots (1-\xi_{k+j-1})\xi_{k+j}.\]  
To begin with, it is easy to notice that 
\[  \left| a_{n,j} - W_j \right| \leq \mathds{1}_{J_n=j} + W_{j,n} + \mathds{1}_{J_n+K_n = j+1}   \]
where \[\left\{\begin{array}{l}
 J_n := \min \{ k\geq 1 : \ \xi_{n-k+1} =1 \} \\ K_n:= \min \{ k\geq 1 : \ \xi_{n+k} =1\} \\ W_{j,n} := \sum_{k=n+1}^{+\infty} \xi_k (1-\xi_{k+1}) \cdots (1-\xi_{k+j-1})\xi_{k+j}
\end{array} \right.\]
therefore \[ \mathbb{E} \left| a_{n,j} - W_j \right| \leq \mathbb{P} (J_n=j ) + \mathbb{E} (W_{j,n}) + \mathbb{P} (J_n+K_n = j+1 ). \]
We look separately at the three right-hand side terms in this inequality.
 \[\mathbb{P} (J_n = j) = \frac{\theta }{\theta + n - j} \frac{n-j+1}{\theta + n -j +1} \cdots \frac{n-1}{\theta + n-1} = \frac{\theta}{n} \Psi_n (j).\]
Hence, $\sum\limits_{j=1}^n \mathbb{P} (J_n = j) = \frac{\theta}{n} \sum\limits_{j=1}^n \Psi_n (j) = 1$. 
\begin{align*}
\mathbb{E} (W_{j,n}) &= \sum\limits_{k=n+1}^{+\infty} \frac{\theta}{\theta + k -1} \frac{k}{\theta + k} \cdots \frac{k+j-2}{\theta + k+j-2} \frac{\theta}{\theta + k +j -1} \\
	&= \sum\limits_{p=n+j+1}^{+\infty} \frac{\theta}{\theta + p-j -1} \frac{p-j}{\theta + p-j} \cdots \frac{p-2}{\theta + p-2} \frac{\theta}{\theta + p -1}.
\end{align*}
\begin{align*}
&\mathbb{P} (J_n+K_n = j+1 ) \\
&\qquad = \sum\limits_{k=1}^j \frac{\theta}{\theta + n +k-j-1} \frac{n+k-j}{\theta + n+k-j} \cdots \frac{n+k-2}{\theta + n+k-2} \frac{\theta}{\theta + n + k-1} \\
	&\qquad = \sum\limits_{p=n+1}^{n+j} \frac{\theta}{\theta + p-j -1} \frac{p-j}{\theta + p-j} \cdots \frac{p-2}{\theta + p-2} \frac{\theta}{\theta + p -1}.
\end{align*}
In particular, 
\begin{align*}
&\mathbb{E} (W_{j,n}) + \mathbb{P} (J_n+K_n = j+1 ) \\
&\qquad = \sum\limits_{p=n+1}^{+\infty} \frac{\theta}{\theta + p-j -1} \frac{p-j}{\theta + p-j} \cdots \frac{p-2}{\theta + p-2} \frac{\theta}{\theta + p -1} \\
	&\qquad= \theta \sum_{p=n}^{+\infty} \frac{1}{p} \frac{A_{p-j}^{\theta -1}}{A_p^\theta} = \theta \left[   \sum_{p=j}^{+\infty} \frac{1}{p} \frac{A_{p-j}^{\theta -1}}{A_p^\theta} -  \sum_{p=j}^{n-1} \frac{1}{p} \frac{A_{p-j}^{\theta -1}}{A_p^\theta} \right].
\end{align*}
Since, by Lemma~\ref{lem:Psi}, \[\sum_{p=j}^{n-1} \frac{1}{p} \frac{A_{p-j}^{\theta -1}}{A_p^\theta} = \Psi_n (j) \left( \frac{1}{j} - \frac{1}{n} \right) \underset{n\to \infty}{\longrightarrow} \frac{1}{j} \]
it follows \[ \mathbb{E} (W_{j,n}) + \mathbb{P} (J_n+K_n = j+1 ) = \theta \left[ \frac{1}{j} - \Psi_n (j) \left( \frac{1}{j} - \frac{1}{n} \right) \right]. \]
Then,
\begin{align*}
\sum_{j=1}^n \mathbb{E} (W_{j,n}) + \mathbb{P} (J_n+K_n = j+1 ) &= \theta \sum_{j=1}^n \frac{1}{j} - \theta \sum_{j=1}^n \frac{\Psi_n (j)}{j} + \frac{\theta}{n} \sum_{j=1}^n \Psi_n (j) \\
	&= \theta \sum_{j=1}^n \frac{1}{j} - \theta \sum_{j=1}^n \frac{1}{\theta + j -1} \ + \ 1 \\
	&= \theta \sum_{j=1}^n \frac{\theta -1}{j(\theta + j -1)} \ + \  1
\end{align*}
using Lemma~\ref{lem:EqPsi} for the second equality. We deduce 
\begin{align}
\begin{split}
\mathbb{E} \left( \sum_{j=1}^n  \left| a_{n,j} - W_j \right|  \right) &\leq 2 + \theta (\theta -1) \sum_{j=1}^n \frac{1}{j(\theta + j-1)} \\
	&\leq 2 + \theta (\gamma + \psi (\theta))
\end{split}
\end{align}
where $\gamma$ is the Euler-Mascheroni constant, and $\psi$ is the digamma function. As a particular consequence,
\begin{equation}
\limsup\limits_{\theta \to 0^+} \ \sup\limits_{n \geq 1} \mathbb{E} \left( \sum_{j=1}^n  \left| a_{n,j} - W_j \right|  \right) \leq 1.
\end{equation}

\section*{Appendix B}

The following proposition gives the whole possible values for the constant $c_2$ appearing in Proposition~\ref{prop:EspVarX}, in function of $\alpha$ and $\beta$.

\begin{proposition}
Let $p, r$ be integers and $q,s$ positive integers with $p$ and $q$ relatively prime, $r$ and $s$ relatively prime, $r\neq 0$. Then
\[c_2 = \left\{\begin{array}{cl}
 \frac{1}{6} & \text{if $\alpha$ and $\beta$ are irrational and} \\
	&\text{ linearly independent over $\mathbb{Q}$} \\
 \frac{1}{6} + \frac{1}{6q^2} & \text{if $ \alpha = \frac{p}{q}$ and $\beta$ is irrational}  \\
 \frac{1}{6} + \frac{1}{6s^2} & \text{if $\alpha$ is irrational and $\beta = \frac{r}{s}$} \\
 \frac{(2q-1)(q-1)}{6q^2} + \frac{(2s-1)(s-1)}{6s^2} - \frac{2}{qs} \sum_{j=1}^{qs} \lbrace \frac{jp}{q} \rbrace \lbrace \frac{jr}{s} \rbrace & \text{if $\alpha = \frac{p}{q}$ and $\beta = \frac{r}{s}$} \\
  \frac{1}{6} - \frac{\gcd (s,q)^2 }{6srq^2} & \text{if $\alpha$ is irrational and} \\
  	&\text{ $\beta = \frac{p}{q}+\frac{r}{s}\alpha $}.
 \end{array}\right. \]
\end{proposition}

The four first cases have been shown by Wieand, as well as the last case for $s=1$ by the same author in \cite{wieand2000eigenvalue}. We complete her work in this appendix, treating the case $s$ arbitrary.

Therefore, suppose $\alpha$ and $\beta$ irrational numbers which are linearly dependent over $\mathbb{Q}$, say $\beta = \frac{p}{q} + \frac{r}{s} \alpha$, with $s\geq 2$. \\
Let us recall that $c_2 := \lim_{n \to \infty} \frac{1}{n} \sum_{j=1}^n (\{j\beta \} - \{j\alpha \})^2$. Since $\lim_{n \to \infty} \frac{1}{n} \sum_{j=1}^n \{jx \}^2$ is easy to compute for all real numbers $x$, it remains to study what was denoted by $s_3$ in \cite{wieand2000eigenvalue}, defined as
\begin{equation}
s_3 (\alpha , \beta ) := \lim_{n \to \infty} \frac{1}{n} \sum_{j=1}^n \{j\alpha \} \{j \beta \}. 
\end{equation}

We first assume that we have $r>0$. We are going to proceed in two steps.

$\bullet$ \underline{Computation of $s_3 (\alpha , \frac{r}{s}\alpha )$} : \\
Before starting the calculation, it is good to notice that if $\psi := \frac{\alpha}{s}$ then $\psi$ is still irrational and consequently the sequence $(n\psi)_{n\in \mathbb{N}^*}$ is still equidistributed. \\
Thus $s_3 (s\psi , r\psi ) = s_3 (s\alpha , r\alpha )$, \emph{i.e} $s_3 (\alpha , \frac{r}{s}\alpha ) = s_3 (s\alpha , r\alpha ) $. Indeed, it is a direct consequence of Theorem 9 given in \cite{wieand2000eigenvalue} that we recall here:  

\begin{theorem}\label{th:Riemann}
Let $f$ be a Riemann integrable function on $[0,1]$, and let $t\in \mathbb{R}\setminus \mathbb{Q}$. Then for all real numbers $b$, 
\[\lim_{n\to \infty} \frac{1}{n} \sum_{j=1}^n f(\{jt + b \} ) = \int_0^1 f(x) \mathrm{d}x.\]
\end{theorem} 

(We take $b=0$ and $f(x)=\{sx\} \{rx\}$.) \\
The expression of $s_3 (s\alpha , r\alpha )$ is more convenient to handle in practice.
We decompose for $j\geq 1$,
\begin{align*}
\lbrace j s\alpha \rbrace \lbrace j r\alpha \rbrace &= \left( s \lbrace j \alpha \rbrace - \lfloor s \lbrace j \alpha \rbrace \rfloor  \right) \left( r \lbrace j \alpha \rbrace - \lfloor r \lbrace j \alpha \rbrace \rfloor  \right) \\
	&=rs \lbrace j \alpha \rbrace^2 - r \sum_{l =1}^{s-1} \lbrace j \alpha \rbrace \mathds{1}_{ \lbrace j \alpha \rbrace \geq \frac{l}{s}} - s \sum_{k=1}^{r-1} \lbrace j \alpha \rbrace \mathds{1}_{ \lbrace j \alpha \rbrace \geq \frac{k}{r}} \\
	&\qquad + \sum\limits_{\substack{1\leq k \leq r-1 \\ 1\leq l \leq s-1}}  \mathds{1}_{ \lbrace j \alpha \rbrace \geq \max \left(\frac{l}{s} , \frac{k}{r} \right) }
\end{align*}
The three first terms will not pose any difficulty since the limits of their respective means have already been evaluated in \cite{wieand2000eigenvalue}. The novelty holds in the fourth term. For $k \in [\![ 1 , r-1 ]\!]$ and $l \in [\![ 1 , s-1 ]\!]$, let $f_{k,l} (x) :=\mathds{1}_{ x \geq \max \left(\frac{l}{s} , \frac{k}{r} \right) } $. These functions are clearly Riemann-integrable on $[0,1]$, thus
\begin{align*}
\frac{1}{n} \sum_{j=1}^n  \sum\limits_{\substack{1\leq k \leq r-1 \\ 1\leq l \leq s-1}}  \mathds{1}_{ \lbrace j \alpha \rbrace \geq \max \left(\frac{l}{s} , \frac{k}{r} \right) } &\underset{n\to \infty}{\longrightarrow} \sum\limits_{\substack{1\leq k \leq r-1 \\ 1\leq l \leq s-1}}  \int_0^1 f_{k,l} (x) \mathrm{d}x \\
	&= (r-1)(s-1) - \sum\limits_{\substack{1\leq k \leq r-1 \\ 1\leq l \leq s-1}} \max \left(\frac{l}{s} , \frac{k}{r} \right)
\end{align*}
with
\begin{align*}
\sum\limits_{\substack{1\leq k \leq r-1 \\ 1\leq l \leq s-1}} \max \left(\frac{l}{s} , \frac{k}{r} \right)
	&= \sum_{k=1}^{r-1} \left( \left( \sum_{l =1}^{\left\lfloor\frac{sk}{r}\right\rfloor} \frac{k}{r} \right) + \left( \sum_{l = \left\lfloor\frac{sk}{r}\right\rfloor + 1 }^{s-1} \frac{l}{s} \right) \right) \\
	&= \sum_{k=1}^{r-1} \left( \left(  \frac{k}{r}\left\lfloor\frac{sk}{r}\right\rfloor  \right) + \left( \frac{s-1}{2} -\frac{\left\lfloor\frac{sk}{r}\right\rfloor \left( \left\lfloor\frac{sk}{r}\right\rfloor +1\right)}{2s}  \right) \right) \\
	&= \frac{(r-1)(s-1)}{2} + \frac{1}{2s} \sum_{k=1}^{r-1}  \left( \frac{sk}{r} \right)^2 - \frac{sk}{r} - \left\{ \frac{sk}{r} \right\}^2 + \left\{ \frac{sk}{r} \right\} \\
	&=  \frac{(r-1)(s-1)}{2} + \frac{1}{2s} \left( \frac{s^2-1}{r^2} \frac{r(r-1)(2r-1)}{6} - \frac{s-1}{r} \frac{r(r-1)}{2} \right). 
\end{align*}
Once expanded, we get 
\[(r-1)(s-1) - \sum\limits_{\substack{1\leq k \leq r-1 \\ 1\leq l \leq s-1}} \max \left(\frac{l}{s} , \frac{k}{r} \right) = \frac{1}{4} + \frac{rs}{3} - \frac{r}{4} - \frac{s}{4} - \frac{s}{12r} - \frac{r}{12s} + \frac{1}{12sr}.\]
We deduce 
\begin{align*}
s_3 (s\alpha , r\alpha) &= \frac{rs}{3} - r \left(\frac{s}{3} - \frac{1}{4} - \frac{1}{12s} \right) - s \left(\frac{r}{3} - \frac{1}{4} - \frac{1}{12r} \right) \\
	&\qquad  + \left(\frac{1}{4} + \frac{rs}{3} - \frac{r}{4} - \frac{s}{4} - \frac{s}{12r} - \frac{r}{12s} + \frac{1}{12sr} \right) \\
	&= \frac{1}{4} + \frac{1}{12sr}
\end{align*}

$\bullet$ \underline{Computation of $s_3 (\alpha , \frac{p}{q} + \frac{r}{s}\alpha )$} : \\
First of all, the particular case $q=1$ is trivial since $s_3 (\alpha , p + \frac{r}{s}\alpha ) = s_3 (\alpha ,\frac{r}{s}\alpha )$ and then $c_2= \frac{2}{3} - 2 \left( \frac{1}{4} + \frac{1}{12sr} \right) = \frac{1}{6} - \frac{\gcd(s,1)^2}{6sr\times 1^2}$. Thus we assume in all the following that $q\geq 2$. \\
In the same manner as at the beginning of the first step, we notice that $s_3 (\alpha , \frac{p}{q} + \frac{r}{s}\alpha ) = s_3 (s\alpha , \frac{p}{q} + r\alpha )$ (using Theorem~\ref{th:Riemann} with $b=\frac{p}{q}$ and for $f$ a complicated expression that we do not precise here, which is bounded and piecewise continuous so Riemann-integrable). \\
For all $j \geq 1$,  
\[\lbrace j s \alpha \rbrace \lbrace j \frac{p}{q} + j r \alpha \rbrace = \lbrace j s \alpha \rbrace  \lbrace j \frac{p}{q} \rbrace + \lbrace j s \alpha \rbrace \lbrace j r \alpha \rbrace - \lbrace j s \alpha \rbrace \mathds{1}_{  \lbrace j \frac{p}{q} \rbrace + \lbrace j r \alpha \rbrace \geq 1 }. \]
The mean of the first term tends to $s_3 (s \alpha , \frac{p}{q}) = \frac{1}{4} - \frac{1}{4q}$. The mean of the second term tends to $s_3 (s\alpha , r\alpha ) = \frac{1}{4} + \frac{1}{12sr}$ by the first step. It remains to study the third term. In a similar way to what is done in \cite{wieand2000eigenvalue}, it is easy to check that 
\[\frac{1}{q \left\lfloor \frac{n}{q} \right\rfloor }  \sum_{j=1}^{q \left\lfloor \frac{n}{q} \right\rfloor}  \lbrace j s \alpha \rbrace  \mathds{1}_{  \lbrace j \frac{p}{q} \rbrace + \lbrace j r \alpha \rbrace \geq 1 }   \underset{n\to\infty}{=} \frac{1}{n} \sum_{j=1}^n  \lbrace j s \alpha \rbrace \mathds{1}_{  \lbrace j \frac{p}{q} \rbrace + \lbrace j r \alpha \rbrace \geq 1 }  \ + \ o(1). \]
Let $n\in \mathbb{N}^*$. We use the periodicity of the sequence $(\lbrace \frac{kp}{q} \rbrace )$ to write that the left-hand side expression is equal to 
\begin{align*}
&\frac{1}{q \left\lfloor \frac{n}{q} \right\rfloor } \sum_{j=0}^{\left\lfloor \frac{n}{q} \right\rfloor - 1} \sum_{k=1}^q  \lbrace (k+jq)s\alpha \rbrace \mathds{1}_{ \lbrace (k+jq)r\alpha \rbrace \geq 1 - \left\{ \frac{kp}{q} \right\} } \\
&\qquad =  s . \frac{1}{q} \sum_{k=1}^q  \frac{1}{\left\lfloor \frac{n}{q} \right\rfloor} \sum_{j=0}^{\left\lfloor \frac{n}{q} \right\rfloor - 1}   \lbrace (k+jq)\alpha \rbrace \mathds{1}_{ \lbrace (k+jq)r\alpha \rbrace \geq 1 - \left\{ \frac{kp}{q} \right\} }    \\
	&\qquad \qquad - \frac{1}{q} \sum_{k=1}^q  \frac{1}{\left\lfloor \frac{n}{q} \right\rfloor} \sum_{j=0}^{\left\lfloor \frac{n}{q} \right\rfloor - 1} \sum_{m=1}^{s-1} \mathds{1}_{ \left(  \lbrace (k+jq) \alpha  \rbrace \geq \frac{m}{s} \right) \cap \left(  \lbrace (k+jq) r \alpha  \rbrace \geq  1 - \left\{ \frac{kp}{q} \right\}  \right) }
\end{align*}
with, (see \cite{wieand2000eigenvalue}), 
\[\frac{1}{q} \sum_{k=1}^q  \frac{1}{\left\lfloor \frac{n}{q} \right\rfloor} \sum_{j=0}^{\left\lfloor \frac{n}{q} \right\rfloor - 1}   \lbrace (k+jq)\alpha \rbrace \mathds{1}_{ \lbrace (k+jq)r\alpha \rbrace \geq 1 - \left\{ \frac{kp}{q} \right\} }  \underset{n\to\infty}{\longrightarrow} \frac{1}{4} + \frac{1}{12r} - \frac{1}{4q}  - \frac{1}{12rq^2}. \] 

For $k \in [\![ 1 , q]\!]$ and $m \in [\![ 1 , s-1 ]\!]$, let
\[f_{k,m} (x) = \mathds{1}_{ \left( x \geq \frac{m}{s} \right) \cap \left( rx - \sum_{l =1}^{r-1} \mathds{1}_{x \geq \frac{l}{r} } \geq 1 - \left\{ \frac{kp}{q} \right\} \right) }. \]
We compute its integral between $0$ and $1$:
\begin{align*}
\int_0^1 f_{k,m} (x) \mathrm{d}x &= \sum_{l = 0}^{r-1} \int_{\frac{1 - \left\{ \frac{kp}{q} \right\} + l}{r}}^{\frac{l+1}{r}} \mathds{1}_{ \left( x \geq \frac{m}{s} \right)} \mathrm{d}x  \\
	&= 0 + \left( \frac{\left\lfloor \frac{rm}{s} \right\rfloor +1}{r} - \max \left( \frac{m}{s} , \frac{1 - \left\{ \frac{kp}{q} \right\} + \left\lfloor \frac{rm}{s} \right\rfloor }{r} \right) \right) \\
	&\qquad + \sum_{l = \left\lfloor \frac{rm}{s} \right\rfloor +1 }^{r-1} \frac{l +1}{r} - \frac{1 - \left\{ \frac{kp}{q} \right\} + l}{r} \\
	&= \frac{1- \left\{ \frac{rm}{s} \right\}}{r} - \frac{1- \left\{ \frac{kp}{q} \right\} - \left\{ \frac{rm}{s} \right\}}{r} \mathds{1}_{1- \left\{ \frac{kp}{q} \right\} - \left\{ \frac{rm}{s} \right\} \geq 0} \\
	&\qquad + \left\{ \frac{kp}{q} \right\} \left( 1 - \frac{1}{r} - \frac{m}{s} + \frac{1}{r} \left\{ \frac{rm}{s} \right\} \right) \\
	&= \frac{1- \left\{ \frac{kp}{q} \right\} - \left\{ \frac{rm}{s} \right\}}{r} \mathds{1}_{1- \left\{ \frac{kp}{q} \right\} - \left\{ \frac{rm}{s} \right\} < 0} + \left\{ \frac{kp}{q} \right\} \left( 1 - \frac{m}{s} + \frac{1}{r} \left\{ \frac{rm}{s} \right\} \right).
\end{align*}
Thus,
\begin{align*}
\sum_{k=1}^q \sum_{m=1}^{s-1} \int_0^1 f_{k,m} (x) \mathrm{d}x &= - \sum_{k=1}^q \left\{ \frac{kp}{q} \right\} + \sum_{k=1}^q  \sum_{m=0}^{s-1} \int_0^1 f_{k,m} (x) \mathrm{d}x  \\
	&= -\sum_{k=0}^{q-1} \frac{k}{q} + \sum_{k=0}^{q-1} \sum_{m=0}^{s-1} \frac{1 - \frac{k}{q} - \frac{m}{s}}{r} \mathds{1}_{\frac{m}{s} > 1 - \frac{k}{q} } + \frac{k}{q} \left( 1- \frac{m}{s} + \frac{m}{sr}  \right) \\
	&= \frac{q-1}{2} \left( s-1 - \frac{s-1}{2} + \frac{s-1}{2r} \right) + \sum_{k=0}^{q-1} \sum_{m=0}^{s-1} \frac{1 - \frac{k}{q} - \frac{m}{s}}{r} \mathds{1}_{\frac{m}{s} > 1 - \frac{k}{q} } 
\end{align*}
where we get the second equality noting that the numbers $\left\{ \frac{kp}{q} \right\}$ (resp. $\left\{ \frac{mr}{s} \right\}$ ) cycle through some rearrangement of the numbers $0 , \frac{1}{q} ,\cdots , \frac{q-1}{q}$ (resp. $0 , \frac{1}{s} ,\cdots , \frac{s-1}{s}$). \\
Now, it just remains to compute the second term in the last expression, which is equal to
\[\frac{1}{r} \sum_{k= \left\lfloor \frac{q}{s} \right\rfloor + 1}^{q-1} \sum_{m = \left\lfloor s- \frac{ks}{q} \right\rfloor + 1}^{s-1}  \left( 1 - \frac{k}{q} - \frac{m}{s} \right) .\]

Note that the condition $\left\lfloor \frac{q}{s} \right\rfloor + 1 \leq q-1$ is always satisfied when $q\geq 3$ since $s\geq 2 > \frac{3}{2} \geq \frac{q}{q-1}$. When $q=2$, it only fails for $s=2$ (in this case we have an empty sum). For now assume we are not in the case $q=s=2$. We will easily treat this remaining case at the end. We have:
\begin{align*}
\sum_{k= \left\lfloor \frac{q}{s} \right\rfloor + 1}^{q-1} \sum_{m = \left\lfloor s- \frac{ks}{q} \right\rfloor + 1}^{s-1} \left(1 - \frac{k}{q} - \frac{m}{s} \right)  &= \sum_{k= \left\lfloor \frac{q}{s} \right\rfloor + 1}^{q-1}  \left( (s-1)- \left\lfloor s- \frac{ks}{q} \right\rfloor \right) \left( 1 - \frac{k}{q} \right) \\
	&\qquad - \frac{1}{s} \sum_{m = \left\lfloor s- \frac{ks}{q} \right\rfloor + 1}^{s-1}  m \\
	&= \sum_{k= \left\lfloor \frac{q}{s} \right\rfloor + 1}^{q-1} \frac{1}{2} \left( \frac{k}{q} - s \left( \frac{k}{q} \right)^2 \right) \\
	&\qquad + \frac{1}{2s} \left( \left\{ \frac{-ks}{q}  \right\}^2 - \left\{ \frac{-ks}{q}  \right\}  \right).
\end{align*}
In addition, using the facts that for all $x \in \mathbb{R}$, $\{-x\}(1-\{-x\})= \{ x\}(1-\{ x\})$, and for all $k \in [\![ 0 , \left\lfloor \frac{q}{s} \right\rfloor  ]\!]$, $0\leq \frac{ks}{q} <1$, it comes
\begin{align*}
&\sum_{k=0}^{\left\lfloor \frac{q}{s} \right\rfloor } \frac{1}{2} \left( \frac{k}{q} - s \left( \frac{k}{q} \right)^2 \right) + \frac{1}{2s} \left( \left\{ \frac{-ks}{q}  \right\}^2 - \left\{ \frac{-ks}{q}  \right\} \right) \\
&\qquad \qquad = \sum_{k=0}^{\left\lfloor \frac{q}{s} \right\rfloor} \frac{1}{2} \left( \frac{k}{q} - s \left( \frac{k}{q} \right)^2 \right) + \frac{1}{2s} \left( \left(\frac{ks}{q}  \right)^2 - \frac{ks}{q}   \right) = 0.
\end{align*}

Furthermore, denoting $d:=\gcd(s,q)$ and $s=d\tilde{s}$, $q=d \tilde{q}$, 
\begin{align*}
\sum_{k=0}^{q-1} \left\{ \frac{ks}{q} \right\} \left( 1 - \left\{ \frac{ks}{q} \right\} \right) &= d \sum_{k=0}^{\tilde{q}-1} \left\{ \frac{k\tilde{s}}{\tilde{q}} \right\} \left( 1 - \left\{ \frac{k\tilde{s}}{\tilde{q}} \right\} \right) \\
	&= d \sum_{k=0}^{\tilde{q}-1} \frac{k}{\tilde{q}} \left( 1 - \frac{k}{\tilde{q}} \right) \\
	&= \frac{q}{6} - \frac{d^2}{6q}.
\end{align*}
Thus, 
\begin{align*}
\sum_{k= \left\lfloor \frac{q}{s} \right\rfloor + 1}^{q-1} \sum_{m = \left\lfloor s- \frac{ks}{q} \right\rfloor + 1}^{s-1} \left(1 - \frac{k}{q} - \frac{m}{s} \right)  &=   - \frac{q}{12s} + \frac{d^2}{12qs} + \frac{1}{2} \sum_{k=0}^{q-1}  \left( \frac{k}{q} - s \left( \frac{k}{q} \right)^2 \right).
\end{align*}
Expanding the expressions it follows
\begin{align*}
\frac{1}{n} \sum_{j=1}^n  \lbrace js \alpha \rbrace \mathds{1}_{  \lbrace j \frac{p}{q} \rbrace + \lbrace j r \alpha \rbrace \geq 1 } & \underset{n\to \infty}{\longrightarrow} s \left( \frac{1}{4} + \frac{1}{12r} - \frac{1}{4q} - \frac{1}{12rq^2} \right) \\ 
	& \qquad - \frac{1}{q} \left(\frac{1}{4} + \frac{qs}{4} - \frac{q}{4} + \frac{qs}{4r} - \frac{q}{4r} - \frac{s}{4}  - \frac{s}{4r} + \frac{1}{4r} \right)         \\
	& \qquad - \frac{1}{qr} \left(  - \frac{q}{12s} + \frac{d^2}{12sq} -\frac{qs}{6} + \frac{s}{4} - \frac{s}{12q} + \frac{q}{4} - \frac{1}{4} \right) \\
	& \qquad = \frac{1}{4} - \frac{1}{4q} + \frac{1}{12sr} - \frac{d^2}{12srq^2}
\end{align*}
Consequently,
\begin{align*}
s_3 \left(\alpha , \frac{p}{q} + \frac{r}{s} \alpha \right) &= \left( \frac{1}{4} - \frac{1}{4q} \right) + \left( \frac{1}{4} + \frac{1}{12sr} \right) - \left( \frac{1}{4} - \frac{1}{4q} + \frac{1}{12sr} - \frac{d^2}{12srq^2} \right) \\
	&= \frac{1}{4} + \frac{d^2}{12srq^2}.
\end{align*}
It is easy to check that this formula remains true for $q=s=2$, since in this case we have 
\[\sum_{k=1}^q \sum_{m=1}^{s-1} \int_0^1 f_{k,m} (x) \mathrm{d}x = \frac{1}{4}+ \frac{1}{4r}\] and thus
\begin{align*}
s_3 \left(\alpha , \frac{p}{q} + \frac{r}{s} \alpha \right) &= \left( \frac{1}{4} - \frac{1}{4q} \right) + \left( \frac{1}{4} + \frac{1}{12sr} \right) \\
	&\qquad - s \left( \frac{1}{4} + \frac{1}{12sr} - \frac{1}{4q} - \frac{1}{12rq^2} \right)  + \frac{1}{q} \left( \frac{1}{4}+ \frac{1}{4r} \right) \\
	&= \frac{1}{4} + \frac{1}{24r} = \frac{1}{4} + \frac{d^2}{12srq^2}.
\end{align*}

Finally, we assume $r<0$. \\
As we have seen above, the result of $s_3 \left(\alpha , \frac{p}{q} + \frac{r}{s} \alpha \right)$ is not depending on the choice of the irrational $\alpha$, thus one can replace it by $(-s)\alpha$, and then
\begin{align*}
s_3 \left(\alpha , \frac{p}{q} + \frac{r}{s} \alpha \right) &=  s_3 \left(s (-\alpha) , \frac{p}{q} + \vert r \vert \alpha \right) \\
	&= \lim_{n \to \infty } \frac{1}{n} \sum_{j=1}^n (1- \{ js\alpha \} ) \left\{ j \left(  \frac{p}{q} + \vert r \vert \alpha \right)  \right\} \\
	&= \lim_{n \to \infty } \frac{1}{n} \sum_{j=1}^n \left\{ j \left(  \frac{p}{q} + \vert r \vert \alpha \right)  \right\}  - s_3 \left( s\alpha , \frac{p}{q} + \vert r \vert \alpha  \right) \\
	&= \frac{1}{2} - \left( \frac{1}{4} + \frac{d^2}{12s \vert r \vert q^2} \right) \\
	&= \frac{1}{4} + \frac{d^2}{12srq^2}.
\end{align*}

\textbf{Acknowledgements}: The author wishes to thank his PhD advisor Joseph Najnudel for suggesting the problem and helpful comments.

\bibliographystyle{plain}
\bibliography{biblio}

\end{document}